\documentclass[reqno]{amsart}
\usepackage{amsmath,amsthm,amssymb}
\usepackage[breaklinks=true,citecolor=blue,colorlinks,linkcolor= blue]{hyperref}
\usepackage{enumitem}

 \newtheorem{thm}{Theorem}[section]
 \newtheorem{cor}[thm]{Corollary}
 \newtheorem{lem}[thm]{Lemma}
 \newtheorem{prop}[thm]{Proposition}
\theoremstyle{definition}
 \newtheorem{defn}[thm]{Definition}
 \newtheorem{rem}[thm]{Remark}
 \newtheorem{ex}[thm]{Example}
 \theoremstyle{plain}
\newtheorem*{Theorem A}{Theorem A} 
\newtheorem*{Theorem B}{Theorem B} 

\def\C{\mathbb{C}}
\def\N{\mathbb{N}}
\def\R{\mathbb{R}}
\def\Z{\mathbb{Z}}
\let\e\varepsilon
\let\le\leqslant
\let\ge\geqslant
\def\cchi{{\raise 1pt\hbox{$\chi$}}}

\makeatletter                                              
\@ifpackageloaded{hyperref}%
  {\newcommand{\mylabel}[2]
    {\protected@write\@auxout{}{\string\newlabel{#1}{{#2}{\thepage}%
      {\@currentlabelname}{\@currentHref}{}}}}}%
  {\newcommand{\mylabel}[2]
    {\protected@write\@auxout{}{\string\newlabel{#1}{{#2}{\thepage}}}}}
\makeatother

\theoremstyle{plain}
\newtheorem*{theorem A}{Theorem A} 
\newtheorem*{theorem B}{Theorem B} 

\begin{document}

\title[$\Gamma$-supercyclicity of families of translates]
{$\Gamma$-supercyclicity of families of translates in weighted $L^p$-spaces on locally compact groups}

\author{Arafat Abbar}
\address{Arafat Abbar, 
LAMA, Univ Gustave Eiffel, Univ Paris Est Creteil, CNRS, F-77447, Marne-la-Vall\'{e}e, France}
\email{arafat.abbar@univ-eiffel.fr}

\author{Yulia Kuznetsova}
\address{Yulia Kuznetsova, 
University Bourgogne Franche-Comt\'{e}, 16 route de Gray, 25030 Besan\c con, France}
\email{yulia.kuznetsova@univ-fcomte.fr}



\begin{abstract}
Let $\omega$ be a weight function defined on a locally compact group $G$,  $1\le p<+\infty$, $S\subset G$ and let us assume that for any $s\in S$, the left translation operator $T_s$ is continuous from the weighted $L^p$-space $L^p(G,\omega)$ into itself. For a given set $\Gamma\subset\mathbb{C}$, a vector $f\in L^p(G,\omega)$ is said to be $(\Gamma,S)$-dense if the set $\{ \lambda T_sf:\, \lambda\in \Gamma, \,s\in S\}$ is dense in $L^p(G,\omega)$. In this paper, we characterize the existence of $(\Gamma,S)$-dense vectors in $L^p(G,\omega)$ in  terms of the weight and the set $\Gamma$.
\end{abstract}

\maketitle

\section{Introduction}
Let $G$ be a locally compact group with a left Haar measure $\mu$, and let $\omega:G\to\mathbb{R}_+$ be a weight on $G$, that is, a positive locally $p$-integrable function. Consider now the weighted $L^p$-space, $1\le p<+\infty$:
$$
L^p(G,\omega) :=\left\{ f: G\to\C:\, f\,\text{ measurable}, \,
\int_{G}|f(t)|^p\omega(t)^p\,d\mu(t)<\infty \right\}
$$
endowed with the norm
$$
\|f\|_{p,\omega}:=\left( \int_{G}|f(t)|^p\omega(t)^p\,d\mu(t)\right)^{1/p}.
$$
For $s\in G$, the left translation operator $T_s$ is defined by
$$
(T_sf)(t)=f(s^{-1}t),\quad t\in G,\ f\in L^p(G,\omega).
$$
It is known that $T_s$ maps $L^p(G,\omega)$ into itself and is bounded if and only if
 \[M(s):=\underset{t\in G}{\mathrm{ess}\sup}\dfrac{\omega(st)}{\omega(t)}<+\infty,\]
and $M(s)$ is equal to the norm of $T_s$ in this case. Let $S$ be a subset of $G$. We say that $\omega$ is an \textit{$S$-admissible} weight if, for all $s\in S$, $M(s)<+\infty$. In this case the following notion makes sense: A function $f\in L^p(G,\omega)$ is called \textit{$S$-dense} if its $S$-orbit
$$
Orb_S(f):=\{ T_sf:\,s\in S\}
$$
is dense in $L^p(G,\omega)$.

Whether $S$-dense functions exist, depends on $S$ and on the weight. The first criterion of the existence of $S$-dense functions in this context was obtained by H.~Salas in \cite{S1}: In the case $G=\mathbb{Z}$ and $S=\mathbb{N}$, a necessary and sufficient condition is $\liminf\{ \omega(n+q)+\omega(-n+q): n\in \mathbb{N}\}=0$ for all $q\in\N$, independently of $p$.

In \cite{AK}, E. Abakumov and Y. Kuznetsova gave an $S$-density criterion in the case of a general group $G$ and its subset  $S$. The (necessary and sufficient) condition involves series of local $p$-norms of $\omega$, similar to that of Theorem \ref{theorem A} below, and is in general not simplifiable. In the case when the subgroup generated by $S$ is abelian, the condition simplifies to the following: {\it For every compact set $F \subset G$ and any given $\delta>0$, there exists $s \in S$ and a compact set $K \subset F$ such that $\mu(F\setminus K)<\delta$ and $\underset{sK\cup s^{-1}K}{\mathrm{ess}\sup}\omega <\delta$.} This is a generalization of the result of Salas,
as well as that of W. Desch et al. \cite{DSW} ($G=\mathbb{R}$, $S=\mathbb{R}\setminus\mathbb{R}_{+}$) and C-C. Chen \cite{C1} (a single translation operator $T_{s_0}$ on $L^p(G,\omega)$).

An interesting necessary condition for the existence of an $S$-dense vector is that $G$ needs to be second-countable, that is, to have a countable base of topology \cite[Proposition 3]{AK}, which is equivalent to saying that $L^p(G,\omega)$ is separable (Recall that $S$ and thus the orbit of $f$ need not be countable).

If $S$ is the sub-semigroup generated by a single point $s_0\in G$, then $S$-density is equivalent to the hypercyclicity of $T_{s_0}$. Recall that an operator $T$ defined on a Banach space $X$ into itself is \textbf{hypercyclic}, if there exists $x\in X$ such that the orbit of $x$, i.e., $\mathrm{Orb}(x,T):=\{ T^nx:\, n\in\mathbb{N}\}$ is dense in $X$. The result of Salas \cite{S1} is thus a criterion of hypercyclicity of the backward shift operator on $L^p(\mathbb{Z},\omega)$. We refer to \cite{BM, Gr, GP} for more information about hypercyclicity and universality.


A bounded linear operator $T$ on a Banach space $X$ is said to be \textbf{supercyclic} if there exists a vector  $x\in X$ whose projective orbit, i.e., $\mathrm{Orb}(\mathbb{C}x, T):=\{ \lambda T^nx:\, \lambda\in\mathbb{C},\, n\in\mathbb{N}\}$ is dense in $X$. A characterization of supercyclicity of weighted shift operators was also obtained by Salas in \cite{S2}, while M. Matsui et al. \cite{MYT} characterized the
supercyclicity of one-parameter translation semigroups. Moreover, if $S$ is the sub-semigroup generated by $s_0\in G$, then a characterization of supercyclicity of $T_{s_0}$ was also obtained by C-C. Chen \cite{C2}.

For $\Gamma\subset\mathbb{C}$, $\Gamma$-supercyclicity is a recent intermediate notion between hypercyclicity and supercyclicity introduced in \cite{CEM}. An operator $T$ defined on a Banach space $X$ is \textbf{$\Gamma$-supercyclic} if there exists $x\in X$ such that the orbit of $\Gamma x$, i.e., $\mathrm{Orb}(\Gamma x, T):=\{ \lambda T^nx:\, \lambda\in\Gamma,\, n\in\mathbb{N}\}$ is dense in $X$.
Instead of a single operator, one can consider a family of translation operators parametrized by a subset $S\subset G$. On this way we come to the following definition, which can be regarded as $\Gamma$-supercyclicity of a family of translates:
\begin{defn}
 We say that $f\in L^p(G,\omega)$ is \textit{$(\Gamma,S)$-dense} in $L^p(G,\omega)$ if its $(\Gamma,S)$-orbit, i.e.,
\[\mathrm{Orb}_{S}(\Gamma f):=\{\lambda T_sf:\,\lambda\in\Gamma, \,s\in S\},\]
is dense in $L^p(G,\omega)$.
\end{defn}

We have clearly a chain of implications:
\[S\text{-density }\Rightarrow (\Gamma,S)\text{-density }\Rightarrow (\mathbb{C},S)\text{-density}.\]
In particular, if $\Gamma$ is reduced to a non-zero point, then $S$-density coincides with $(\Gamma,S)$-density. Furthermore, if $S$ is the sub-semigroup generated by $s_0\in G$, then $(\Gamma,S)$-density coincides with $\Gamma$-supercyclicity of $T_{s_0}$. 

Besides the case of a trivial group $G=\{ e\}$, the group must be non-compact second-countable in order to have $(\Gamma,S)$-dense vectors for some $p,\omega,S,\Gamma$. This is proved in Section \ref{sec-group}. Under the above assumptions on the group, we get the following density criterion extending the results of \cite{AK}. Denote, for a measurable function $f$ on $G$ and a subset $K$ of $G$,
$$
\| f\|_{p,K}:=\left( \int_{K}|f(t)|^p\, d\mu(t)\right)^{1/p}.
$$
\begin{theorem A}\label{theorem A}
Let $G$ be a second-countable locally compact non-compact group. Let
$S\subset G$, $1\le p<+\infty$, $\Gamma\subset\mathbb{C}$ be such that $\Gamma\setminus\{0\}$ is non-empty,
and let $\omega$ be a locally $p$-integrable $S$-admissible weight on $G$. Then the
following conditions are equivalent:
\begin{enumerate}[label={$(\arabic*)$}]
\item There is a $(\Gamma,S)$-dense vector in $L^p(G,\omega)$.
\item For every increasing sequence of compact subsets  $( F_n)_{n\ge1}$ of $G$ of positive measure and for every sequence  $(\delta_n)_{n\ge1}$ of positive numbers, there are sequences $(s_n)_{n\ge1}\subset S
$, $(\lambda_n)_{n\ge1}\subset\Gamma\setminus\{0\}$ and a sequence of compact subsets $K_n\subset F_n$ such that the sets $s_{n}^{-1}F_n$ are pairwise disjoint, $\mu(F_n\setminus K_n)<\delta_n$ and
\[
\sum_{n,k\ge0;n\neq k}\dfrac{|\lambda_n|^p}{|\lambda_k|^p} 
\|\omega\|_{p,s_ns_{k}^{-1}K_{k}}^p<+\infty,
\]
with $s_0=e$ being the identity, $\lambda_0=1$ and $K_0=\emptyset$.
\end{enumerate}
\mylabel{theorem A}{A}
\end{theorem A}
This theorem is stated again, with yet another equivalent condition, and proved in Section \ref{sec-proof-A}.

In \cite{Sh}, S.~Shkarin proved the equivalence of $\mathbb{R}_{+}$-supercyclicity and supercyclicity. By Theorem \ref{theorem A}, we obtain the following generalization of this result as a corollary:
\begin{cor}
Let $|\Gamma|$ denote the set $\{ |\lambda|:\,\lambda\in\Gamma\}$. $L^p(G,\omega)$ has a $(\Gamma,S)$-dense vector if and only if it has a $(|\Gamma|,S)$-dense vector. In particular,
$L^p(G,\omega)$ has an $(\mathbb{R}_{+},S)$-dense vector if and only if it has a $(\mathbb{C},S)$-dense vector.
\end{cor}
Let $T$ be an operator on a Banach space $X$. If \;$\Gamma\subset\mathbb{C}$ is such that $\Gamma\setminus\{0\}$ is non-empty, bounded and bounded away from zero, we know that $T$ is $\Gamma$-supercyclic if and only if $T$ is hypercyclic \cite{CEM}. Here we have the following similar result for a family of translation operators:
\begin{cor}
If\/ $\Gamma\setminus\{0\}$ is bounded and bounded away from zero, then $L^p(G,\omega)$ has an $S$-dense vector if and  only if it has a $(\Gamma,S)$-dense vector. 
\end{cor}

If $\Gamma$ is unbounded or zero is its limit point, then $L^p(G,\omega)$ may have a $(\Gamma,S)$-dense vector, but no $S$-dense vectors, see Example \ref{no-S-dense}. 

If the subgroup generated by $S$ is abelian, Theorem \ref{theorem A} simplifies to the following, under the same assumptions on $p,G,S,\Gamma$ and $\omega$ as in Theorem \ref{theorem A}:
\begin{theorem B}\label{theorem B}
 If the subgroup generated by $S$ is abelian, then the following conditions are equivalent:
\begin{enumerate}[label={$(\arabic*)$}]
\item There is a $(\Gamma,S)$-dense vector in  $L^p(G,\omega)$.
\item For any compact subset $F\subset G$ and $\e>0$, there are $s\in S$, $\lambda\in\Gamma\setminus\{0\}$ and a compact subset $E\subset F$ such that $\mu(F\setminus E)<\e$ and
\[|\lambda|\underset{t\in E}{\mathrm{ess }\sup }\,\omega(st)<\e\quad\text{ and }\quad \dfrac{1}{|\lambda|}\underset{t\in E}{\mathrm{ess }\sup}\,\omega(s^{-1}t)<\e.\]
\end{enumerate}
\mylabel{theorem B}{B}
\end{theorem B}
For the proof of Theorem \ref{theorem B}, with an additional equivalent condition, see Section \ref{sec-abelian}. It is worth mentioning that the condition (2) above does not depend on $p$, but we need the usual assumption of local $p$-summability of $\omega$ which does depend on $p$.

In the case where $\Gamma$ is equal to one of these sets: the complex plane, $[0,1]$, $[1,+\infty[$, or a set which is bounded and bounded away from zero, we get a complete characterization of $(\Gamma,S)$-density only in terms of the weight, see Corollary \ref{26}.

And finally, in the case when all translations are bounded (and not only by $s\in S$), the criterion is as follows:
\begin{prop}
Let $G$ be abelian and let $\omega$ be a continuous $G$-admissible weight. Then there is a $(\Gamma,S)$-dense vector in $L^p(G,\omega)$ if and only if there exist two sequences $(s_n)_{n\ge1}\subset S$ and $(\lambda_{n})_{n\ge1}\subset\Gamma\setminus\{0\}$ such that
\[\underset{n\to+\infty}{\lim}\max\left\{  |\lambda_n|\omega(s_n)\,;\,\frac{1}{|\lambda_n|}\omega(s_n^{-1})\right\} =0.\]
\label{3}
\end{prop}
Proposition \ref{3} is a consequence of Proposition \ref{MR} below when $G$ is abelian. In the case where $G$ is abelian, and $\Gamma$ is equal  to one the particular mentioned above sets,  we get a complete characterization of $(\Gamma,S)$-density only in terms of the weight, see Corollary \ref{29}. 

\section{Restrictions on the group}\label{sec-group}

It has been shown in \cite{AK} that $S$-dense vectors can exist in $L^p(G,\omega)$ only if $G$ is non-compact second-countable. For $(\Gamma,S)$-density, the restrictions are the same, even if the proofs have to be changed.

In this section and further on we denote by $\cchi_K$ the characteristic function of a set $K$, and by $e$ the identity of $G$.

\begin{lem}
Let $X$ be a non-separable Banach space and let $0<M<1$. Then there exist an uncountable set $\{ x_\alpha\}_{\alpha\in I}\subset X$ of norm 1 vectors such that for all $\alpha_1,...,\alpha_n\in I$ and $c_1,...,c_n\in\mathbb{C}\setminus\{0\}$,
$$
\|\sum_{i=1}^{n}c_ix_{\alpha_i}\|> M\min\{|c_i|:\, 1\le i\le n\}.
$$
\label{8}
\end{lem}
\begin{proof}
This is an easy application of Zorn's lemma and Riesz's lemma: note that $\|\sum_{i=1}^{n}c_ix_{\alpha_i}\|\ge |c_k|{\rm dist}(x_{\alpha_k}, {\rm span}\{x_{\alpha_i}: k\ne i\in\{1,\dots,n\}\}$ for every $k=1,\dots, n$.
\end{proof}

\begin{lem}\label{lem-tam}
Let $G$ be a $\sigma$-compact locally compact group. The orbit ${\rm Orb}_G(f)$ of any function $f$ in $L^p(G)$ is separable (in the norm topology).
\end{lem}
\begin{proof}
This fact is not new, and proved for example in \cite{tam} for $p=1$. Since it is not stated for $p>1$ up to our knowledge, we include a proof here. It is known \cite[Theorem 20.4]{HR} that the map $\tau:s\mapsto T_sf$ is continuous, from $G$ to $L^p(G)$ with the norm topology. Since $G$ is a countable union of compact sets, the same holds for $\tau(G) = {\rm Orb}_G(f)$. This is in addition a metric space, which implies that it is separable.
\end{proof}

\begin{thm}
If for some $p$, $S$, and $\omega$ there exists a $(\mathbb{C},S)$-dense vector in $L^p(G,\omega)$, then $G$ is second-countable which is equivalent to saying that $L^p(G,\omega)$ is separable. 
\end{thm}
\begin{proof}
Let us prove first that $G$ is $\sigma$-compact, which is a strictly weaker condition. Suppose that $f$ is a $(\mathbb{C},S)$-dense vector in $L^p(G,\omega)$ for some $p,\omega$ and $S$. It is known that $f$ can be assumed to vanish outside a $\sigma$-compact set \cite[Theorem 11.40]{HR}, so that the support of $f$ is contained in an open $\sigma$-compact subgroup $H$ of $G$.

Suppose now that $G$ is not $\sigma$-compact, then $H\ne G$. Pick $g\in G\setminus H$. Let $U\subset H$ be a compact symmetric neighbourhood of identity. There exist $\gamma\in\C$ and $s\in S$ such that $\|\cchi_U+\cchi_{gU}-\gamma T_sf\|_{p,\omega}< \min( \|\cchi_U\|_{p,\omega}, \|\cchi_{gU}\|_{p,\omega})$. But this is only possible if the support of $T_sf$ intersects both $H$ and $gH$, which is false by the choice of $H$.

This shows that from now on, we can assume that $G$ is $\sigma$-compact. According to \cite[Theorem 2]{V}, we have $G$ is second-countable if and only if $L^p(G)$ is separable, and the same is valid for $L^p(G,\omega)$ since this space is isometrically isomorphic to $L^p(G)$.

Suppose that $L^p(G,\omega)$ is not separable but $f\in L^p(G,\omega)$ is a $(\mathbb{C},S)$-dense vector. According to Lemma \ref{8}, there exists an uncountable set $\{ g_\alpha\}_{\alpha\in I}\subset L^p(G,\omega)$ such that $\| g_\alpha\|_{p,\omega}=1$ for all $\alpha$ and 
$$
\|\sum_{i=1}^{n}c_i g_{\alpha_i}\|_{p,\omega}> \dfrac{9}{10}\min\{|c_i|:\, 1\le i\le n\}
$$
for all $\alpha_1,...,\alpha_n\in I$ and $c_1,...,c_n\in\mathbb{C}\setminus\{0\}$. For all $\alpha\in I$, we can approximate $g_\alpha$ so that $\| g_\alpha-\lambda_\alpha T_{s_\alpha}f\|_{p,\omega}<\dfrac{1}{10}$, with $s_\alpha\in S$ and $\lambda_\alpha\in\mathbb{C}\setminus\{0\}$. We can choose $\delta>0$ so that the set
$J=\{\alpha\in I: \delta<|\lambda_\alpha|<2\delta\}$ is uncountable. Now for every $\alpha,\beta\in J$ we have
\[\| \lambda_{\alpha}^{-1} g_\alpha- T_{s_\alpha}f\|_{p,\omega}<\dfrac{1}{10|\lambda_\alpha|}<\dfrac{1}{10\delta},\]
and
$$
\|T_{s_\alpha}f-T_{s_\beta}f\|_{p,\omega}
 \ge \| \frac1{\lambda_{\alpha}} g_\alpha-\frac1{\lambda_\beta} g_\beta\|_{p,\omega} -\frac1{5\delta}
 \ge\dfrac{9}{10} \min\{ \frac1{|\lambda_\alpha|}, \frac1{|\lambda_\beta|}\} -\frac1{5\delta}
  > \frac1{4\delta}.
$$
We will show next that the set $J$ allows to find a function in $L^p(G)$ with a non-separable orbit, which contradicts Lemma \ref{lem-tam} and will prove the theorem. For every $\alpha\in I$, we have
$$
\|T_{s_\alpha}f\|_{p,\omega}^p = \lim_{N\to+\infty} \int_{\{x:\omega(x)\le N\}} |T_{s_\alpha}f(x)|^p\omega(x)^pd\mu(x)
$$
so that there exists $N_\alpha\in\N$ such that $\int_{\{x:\omega(x)> N_\alpha\}} |T_{s_\alpha}f(x)|^p\omega(x)^pdx < (16\delta)^{-1}$. There is $N\in\N$ and an uncountable set $J_0\subset J$ such that $N_\alpha=N$ for every $\alpha\in J_0$. Moreover, there exists $C>0$ and an uncountable set $J_1\subset J_0$ such that $\|T_{s_\alpha}\|\le C$ for every $\alpha\in J_1$.

Next, for every $\e>0$ the function
$$
f_\e(x) = f(x) \cchi_{\{\omega\ge \e\}} (x)
$$
is in $L^p(G)$. Pick $\e>0$ such that $\|f_\e-f\|_{p,\omega}<(32C\delta)^{-1}$.
We have now for $\alpha,\beta\in J_1$:
\begin{align*}
\|T_{s_\alpha}f_\e - T_{s_\beta}f_\e\|_p &\ge \frac1N \|\cchi_{\omega\le N}\cdot\big[T_{s_\alpha}f_\e - T_{s_\beta}f_\e\big]\|_{p,\omega}
\\& \ge \frac1N\Big[ \|T_{s_\alpha}f - T_{s_\beta}f\|_{p,\omega} - \|\cchi_{\omega> N}\cdot\big[T_{s_\alpha}f - T_{s_\beta}f\big]\|_{p,\omega}
\\& - \|\cchi_{\omega\le N}\cdot\big[T_{s_\alpha}(f_\e-f) - T_{s_\beta}(f_\e-f)\big]\|_{p,\omega} \Big]
\\& \ge \frac1N\Big[ \frac1{4\delta} - \frac2{16\delta} - \|T_{s_\alpha}(f_\e-f)\|_{p,\omega} - \|T_{s_\beta}(f_\e-f)\|_{p,\omega} \Big]
\\& \ge \frac1N\Big[ \frac1{8\delta} - 2C\|f_\e-f\|_{p,\omega} \Big] > \frac1{16\delta N}>0.
\end{align*}
This means that the $S$-orbit of $f_\e$ is nonseparable in $L^p(G)$. This is the expected contradiction, which proves the theorem.
\end{proof}


\begin{prop}
If $G$ is a nontrivial compact group, then the space $L^p(G,\omega)$ has no $(\Gamma,S)$-dense vectors for any $\Gamma,\omega,S,p$.
\end{prop}
\begin{proof}
Suppose that $L^p(G,\omega)$ has a $(\Gamma,S)$-dense vector $f$. If $G$ is finite, then $\mathrm{Orb}_{S}(\Gamma f)$ is contained in the finite union $\cup\{ \C T_gf: g\in G\}$ of one-dimensional lines, which cannot be dense in $L^p(G,\omega)$ unless $G=\{e\}$. We suppose therefore that $G$ is infinite.

Scale the Haar measure so that $\mu(G)=1$. Fix $\e\in(0,1/12)$. There is $\delta>0$ such that $\mu(E_\delta)>1-\e$ where $E_\delta = \{t\in G: \omega(t)>\delta\}$. There exist $s_1\in S$ and $\gamma_1\in\Gamma$ such that $\|\cchi_G-\gamma_1 T_{s_1}f\|_{p,\omega}<\e \delta$. Then
$$
\delta^p\e^p > \int_{E_\delta} \delta^p |1-\gamma_1 f(s_1^{-1}t)|^p d\mu(t),
$$
so that
$$
\int_{s_1^{-1} E_\delta} |1-\gamma_1f(t)|^p d\mu(t) < \e^p.
$$

Since $G$ is infinite, $\mu$ has no atoms. We know that $G$ is second-countable, thus $e$ has a countable base of neighbourhoods $(U_n)$ and one can assume them decreasing; there exists $U_n$ with $\mu(U_n)<\e$. By compactness, $G$ is covered by a finite number of translates of $U_n$ (all of the same measure $<\e$); picking a sufficient number of them, one can form a set $E\subset G$ such that $1/2-\e<\mu(E)<1/2+\e$.
Let now $s_2\in S$, $\gamma_2\in\Gamma$ be such that $\|\cchi_{G\setminus E}-\gamma_2 T_{s_2}f\|_{p,\omega}<\delta\e$. As above, this implies 
$$
\delta^p\e^p > \int_{E_\delta\cap E} \delta^p |\gamma_2|^p |f(s_2^{-1}t)|^p d\mu(t),
$$
so that
$$
\int_{s_2^{-1} (E_\delta\cap E)} |\gamma_2 f(t)|^p d\mu(t) < \e^p,
$$
and
$$
\int_{ s_2^{-1} (E_\delta\setminus E)} |1-\gamma_2 f(t)|^p d\mu(t) < \e^p.
$$
Denote $A= s_1^{-1} E_\delta$, $B= s_2^{-1} (E_\delta\cap E)$, $C= s_2^{-1} (E_\delta\setminus E)$.
Now $\mu( A \cap B) \ge 1-\mu( G\setminus A) - \mu( G\setminus B) > 1-\e - 1/2-2\e = 1/2-3\e$, and similarly $\mu( A \cap C) \ge 1/2-3\e$. At the same time (the norms below are unweighted),
\begin{align*}
\mu(A\cap B)^{1/p} &= \|\cchi_{A\cap B}\|_p
 \le  \|\cchi_{A\cap B} \gamma_1 f\|_p + \|\cchi_{A\cap B} (1-\gamma_1f)\|_p < \frac{ |\gamma_1|}{|\gamma_2|} \e + \e,
\end{align*}
so if we denote $z=\gamma_1/\gamma_2$ then $\e|z| > (1/2-3\e)^{1/p}-\e \ge 1/2-4\e > 1/6$ by the choice of $\e$ and $|z|>1/(6\e)$. In particular, $|z|>2$. On the other hand,
\begin{align*}
\mu(A\cap C)^{1/p}\Big| 1-\frac{\gamma_1}{\gamma_2}\Big| &= \|\cchi_{A\cap C}\Big(1-\frac{\gamma_1}{\gamma_2}\Big)\|_p
 \\&\le  \|\cchi_{A\cap C} (1-\gamma_1 f)\|_p + \|\cchi_{A\cap C} \frac{\gamma_1}{\gamma_2}(1-\gamma_2f)\|_p
 < \e+ \frac{ |\gamma_1|}{|\gamma_2|} \e,
\end{align*}
whence similarly $|1-z|/4<\e(1+|z|)$ and $|z| < 1+4\e(1+|z|)$ which implies $|z|<(1+4\e)/(1-4\e)$. This is however incompatible with $|z|>2$.
This contradiction proves the proposition.
\end{proof}

\section{Continuity of translations}

The condition for the translation operator $T_s$ to be continuous was given in the introduction. The question of continuity of the map $S\to L^p(G,\omega)$, $s\mapsto T_sf$ for fixed $f$ is more delicate.

Let $C_{c}(G)$ denote the space of continuous functions on $G$ with compact support. On this space, no problem occurs. In a standard way (similarly to \cite[Theorem (20.4)]{HR}, using uniform continuity of $f$ and local $p$-summability of $\omega$), one proves the following
\begin{prop}
Let $G$ be a locally compact group, $1\le p<+\infty$, and let $\omega$ be a $G$-admissible weight on $G$. Then for every $f\in C_c(G)$ the map
 \[ \begin{array}{lrcl}
   & G& \longrightarrow & L^p(G,\omega)\\
        & s & \longmapsto & T_sf \end{array}\]
is continuous.
\label{4}
\end{prop}

If $\omega$ is a $G$-admissible weight (all translations are bounded), the map in question is still continuous for any $f\in L^p(G,\omega)$:
\begin{thm}
Let $G$ be a locally compact group, $1\le p<+\infty$, and let $\omega$ be a $G$-admissible weight on $G$. Then for every $f\in L^p(G,\omega)$ the map
 \[
  \begin{array}{lrcl}
   & G& \longrightarrow & L^p(G,\omega)\\
        & s & \longmapsto & T_sf \end{array}       
\]
is continuous.
\label{16}
\end{thm}
\begin{proof}
It is known \cite[Theorem 2.7]{feicht} that every $G$-admissible locally summable weight $\omega$ is equivalent to a continuous weight $\tilde\omega$, in the sense that there exists $C>0$ such that $1/C\le \omega(t)/\tilde\omega(t)\le C$ for all $t\in G$. For $\omega$ locally $p$-summable the same is true of course. We can thus assume that $\omega$ is continuous.

Set $M(s):=\underset{t\in G}{\sup}\,\dfrac{\omega(st)}{\omega(t)}$ for every $s\in G$. According to \cite[Proposition 1.16]{E}, $M(s)$ is locally bounded on $G$. Fix $f\in L^p(G,\omega)$. Let $s_{0}\in G$ and $\e>0$. Since $G$ is locally compact, there exists an open neighborhood $W$ of $e$ whose closure is compact. Let $\delta>0$ such that
\[
\big(\|T_{s_{0}}\|+1+\underset{s\in s_{0}\overline{W}}{\sup}M(s)\big)\delta<\e.
\]
Since $\omega$ is continuous, the space $C_c(G)$ is dense in $L^p(G,\omega)$, hence, there exists $\phi\in C_{c}(G)$ such that
\[\|f-\phi\|_{p,\omega}<\delta.\]
By Proposition \ref{4}, there exists a symmetric open neighborhood $U$ of $e$  whose closure is compact such that
\[\|T_{s_{0}}\phi-T_s\phi\|_{p,\omega}<\delta\, \text{ for every } s\in s_{0}U .\]
Let $V=U\cap W$, then $V$ is an open neighborhood of $e$. Fix $s\in s_{0}V$. 
Thus  $\|T_{s_{0}}\phi-T_s\phi\|_{p,\omega}<\delta$ and $\|T_s\phi-T_sf\|_{p,\omega}<\delta\underset{s\in s_0\overline{W}}{\sup}M(s)$.
Hence
\begin{align*}
\|T_{s_{0}}f-T_{s}f\|_{p,\omega}&\le\|T_{s_{0}}f-T_{s_{0}}\phi\|_{p,\omega}+\|T_{s_{0}}\phi-T_s\phi\|_{p,\omega}+\|T_s\phi-T_sf\|_{p,\omega}\\
&<(\|T_{s_{0}}\|+1+\underset{s\in s_{0}\overline{W}}{\sup}M(s))\delta<\e.
\end{align*}
\end{proof}

With a weight as above, it is not difficult to show the following ``non-compactness'' of translations, necessary in the sequel:

\begin{prop}
Let $G$ be a locally compact non-trivial group, $S\subset G$, $1\le p<+\infty$, and let $\Gamma\subset\mathbb{C}$ be such that $\Gamma\setminus\{0\}$ is non-empty. If $\omega$ is a \textbf{$G$-admissible} weight on $G$ and $f\in L^p(G,\omega)$ is a $(\Gamma,S)$-dense vector, then for every compact set $K\subset G$,  $f$ is also a $(\Gamma,S\setminus K)$-dense vector in $L^p(G,\omega)$.
\end{prop}
\begin{proof}
As pointed out above, we can assume that $\omega$ is continuous. Since $G$ is not a finite group, $L^p(G,\omega)$ is infinite-dimensional. By Theorem \ref{16}, the set $\{ T_sf:\, s\in K\}$ is compact in $L^p(G,\omega)$, thus the set $\mathbb{C}\{ T_sf:\, s\in K\}$ is nowhere dense. Hence
\[L^p(G,\omega)=\mathrm{int}\,\overline{\mathrm{Orb}_{S}(\Gamma f)\setminus \mathbb{C}\{ T_sf:\, s\in K\} }=\mathrm{int}\,\overline{\mathrm{Orb}_{(S\setminus K)}(\Gamma f)},\]
which implies that $f$ is a $(\Gamma,S\setminus K)$-dense vector in $L^p(G,\omega)$.
\end{proof}

However, if translations $T_s$ are bounded only for $s$ in a subset $S\subset G$, then the translation map $s\mapsto T_sf$ may be discontinuous on $S$. This explains the need of the long Lemma \ref{noncompact} below.

In the following example, we construct a continuous weight on $\R$, prove first that it is $\mathbb R_+$-admissible, and then exhibit a function $f\in L^p(\R,\omega)$ such that $T_s f\not\to f=T_0f$ as $s\to0$.

\begin{ex}
Set $G=\mathbb{R}$, $S=\mathbb{R}_{+}$ and $a_n=n!$ for $n\ge2$. Define the weight as
\[
   \omega(t):= \begin{cases}
    1 &\text{ if } t\le 0 \text{ or } t\in[a_{n-1}+1,a_n-n], n\ge 3\\
    2^{t-(a_n-n)} &\text{ if } t\in[a_n-n,a_n-1] \\
    2^{n-1} &\text{ if } t\in[a_n-1,a_n-1/2] \\
    1+(2^n-2)(a_n-t) &\text{ if } t\in[a_n-1/2,a_n]\\
    1+(2^n-1)2^n(t-a_n) &\text{ if } t\in[a_n,a_n+2^{-n}] \\
    \dfrac{1}{t-a_n} &\text{ if } t\in[a_n+2^{-n},a_n+1]
    \end{cases}
\] 
\end{ex}
This weight is thus equal to 1 at $a_n$ and ``far from'' any $a_n$; around $a_n$, it has two peaks: one on the left, with the maximum value $2^{n-1}$ taken on a whole segment $[a_n-1,a_n-1/2]$, and another on the right, with the maximum $2^n$ taken in $a_n+2^{-n}$. It is clear that $\omega$ is continuous.

{\bf Claim 1.} $\omega$ is an $\mathbb{R}_{+}$-admissible weight.
\begin{proof}
Set 
\[M(s)=\underset{t\in\mathbb{R}}{\sup}\dfrac{\omega(s+t)}{\omega(t)}.\]
Suppose that $0<s<1/2$, and let $n$ be such that $2^{-n-1}<s\le 2^{-n}$. Clearly
\[
M(s)\ge\dfrac{\omega(a_n+s)}{\omega(a_n)}=1+(2^{2n}-2^n)s \ge 2^{2n-1}s \ge \frac1{8s}.
\]
Let us show that in fact $M(s)\le 2/s$. To estimate $M(s)$, let us consider all possible cases for $t\in\R$.
\begin{itemize}
\item If $t\le a_{n+1}-(n+1)$, then $\omega(t)$ and $\omega(t+s)$ are  between $1$ and $2^n$, so that their ratio is bounded by $2^n$. 
\item If $t\in[a_m-m,a_m-1-s]$ for $m>n$, then
\[
\dfrac{\omega(t+s)}{\omega(t)}=2^s\le 2.
\] 
\item If $t\in[a_m-1-s,a_m-1/2]$ for $m>n$, then
\[
\dfrac{\omega(t+s)}{\omega(t)}\le\dfrac{2^{m-1}}{2^{m-1-s}}=2^s\le 2.
\]
\item If $t\in[a_m-1/2,a_m-s]$ for $m>n$, then $\omega(t+s)\le\omega(t)$ and 
\[
\dfrac{\omega(t+s)}{\omega(t)}\le1.
\]
\item If $t\in[a_m-s,a_m+2^{-m}-s]$ for $m>n$, denote $u=a_m-t$. We have $u\in[s-2^{-m},s]$, and 
\begin{align*}
\dfrac{\omega(t+s)}{\omega(t)}&=\dfrac{1+(2^m-1)2^m(t+s-a_m)}{1+(2^m-2)(a_m-t)}\\
&=\dfrac{1+(2^m-1)2^m(s-u)}{1+(2^m-2)u}=:P(u).
\end{align*}
$P$ is a decreasing function of $u$ on $[s-2^{-m},s]$, hence
\begin{align*}
\dfrac{\omega(t+s)}{\omega(t)}&\le P(s-2^{-m})=\dfrac{2^m}{2^m s - 2s +2^{-m+1}}\\
&\le\dfrac{2^m}{(2^m-2)s}\le \frac1{2s}.
\end{align*}
\item If $t\in[a_m-s+2^{-m},a_m]$ for $m>n$, denote again $u=a_m-t$. We have $u\in[0,s-2^{-m}]$, and 
$$
\dfrac{\omega(t+s)}{\omega(t)}=\dfrac{1}{[1+(2^m-2)(a_m-t)](t+s-a_m)}=\dfrac{1}{Q(u)},
$$
where $Q(u):=(1+(2^m-2)u)(s-u)$. It is easy to see that $Q$ is a concave function, and thus its minimum on any segment is attained at one of its ends. On $[0,s-2^{-m}]$, we have $Q(0)=s$ and
\[
Q(s-2^{-m})=s+2^{1-2m}-2^{1-m}s\ge (1-2^{1-m})s\ge \frac s2,
\]
as we suppose $m>n\ge1$. It follows that
\[
\dfrac{\omega(t+s)}{\omega(t)}=\dfrac{1}{Q(u)}\le \frac2s.
\]
\item If $t\in[a_m,a_m+2^{-m}]$ with $m>n$, denote $u=t-a_m\in[0,2^{-m}]$. We have
$$
\dfrac{\omega(t+s)}{\omega(t)}=\dfrac{1}{(t+s-a_m)(1+(2^m-1)2^m(t-a_m))}=R(u),
$$
where $R(u):=\dfrac{1}{(s+u)((2^{2m}-2^m)u+1)}$. Then $R$ is decreasing on $[0,2^{-m}]$, hence
\[
\dfrac{\omega(t+s)}{\omega(t)}\le R(0)=\frac1s.
\]
\item If $t\in[a_m+2^{-m},a_m+1]$, then $\omega(t+s)\le\omega(t)$.
\end{itemize}
Taking maximum of these estimates, we arrive at $M(s)\le\dfrac2s$.

Being finite on $[0,1/2)$ and submultiplicative (i.e., $M(s+s')\leqslant M(s)M(s')$), $M$ is finite on $[0,\infty)$ so that $\omega$ is an $\R_+$-admissible weight.
\end{proof}

One can verify that $M(s)<\infty$ exactly when $s\in\R_+$: if $s\le -1$, then
$$
M(s)\ge \sup_{n\ge 2}\dfrac{\omega(a_n+2^{-n})}{\omega(a_n+2^{-n}+|s|)}\ge \sup_{n\ge 2} 2^n=+\infty,
$$
and if $-1<s<0$ then
$$
M(s)\ge \sup_{n: 2^{-n}\le 1-|s|}\dfrac{\omega(a_n+2^{-n})}{\omega(a_n+2^{-n}+|s|)}= \sup_{n: 2^{-n}\le 1-|s|} 2^n (2^{-n}+|s|) =+\infty.
$$

{\bf Claim 2.} There exists $f\in L^p(\R,\omega)$ such that $\|f-T_sf\|_{p,\omega}\not\to0$, $s\to0$.
\begin{proof}
Set $f = \sum_{k=2}^\infty \cchi_{[a_k-2^{-k}, a_k]}$. Check first that its $p,\omega$-norm is finite:
\begin{align*}
\|f\|_{p,\omega}^p &= \sum_{k=2}^\infty \int_{a_k-2^{-k}}^{a_k} \big( 1+(2^k-2)(a_k-t)\big)^p dt
\\&\le \sum_{k=2}^\infty 2^{-k} (1+ (2^k-2) 2^{-k})^p <\infty.
\end{align*}
If $s=2^{-n}$, $2\le n\in\N$, then $T_sf=\sum_{k=2}^\infty \cchi_{[a_k-2^{-k}+2^{-n}, a_k+2^{-n}]}$; it equals 1 on $[a_n, a_n+2^{-n}]$ while $f$ vanishes (almost everywhere) on this segment. Thus
\begin{align*}
\|f-T_sf\|_{p,\omega}^p &\ge \int_{a_n}^{a_n+2^{-n}} \omega(t)^p dt
= \int_{a_n}^{a_n+2^{-n}} \big(1+(2^n-1)2^n(t-a_n)\big)^p dt
\\&\ge 2^{p(n-1)+pn} \int_0^{2^{-n}} t^p dt = \frac1{(p+1)2^p} 2^{2pn} 2^{-n(p+1)} = \frac1{(p+1)2^p} 2^{n(p-1)}
\end{align*}
which is constant if $p=1$ and tends to infinity as $n\to \infty$ if $p>1$.
\end{proof}

The following lemma will be used in the proof of Theorem \ref{theorem A}. 
\begin{lem}\label{noncompact}
Let $G$ be a locally compact non-compact group, $S\subset G$, $1\le p<+\infty$, and $\omega$ an $S$-admissible weight. Let $f$ be a $(\Gamma,S)$-dense vector in $L^p(G,\omega)$. Then for every compact set $K\subset G$ of positive measure, any $\e>0$, $\lambda\in\C\setminus\{0\}$ and every compact set $L\subset G$ there are $\gamma\in\Gamma$, $s\in S\setminus L$ such that $$\|\gamma T_sf-\lambda \cchi_{K}\|_{p,\omega}<\e.$$ 
\label{9}
\end{lem}
\begin{proof}
We can scale the Haar measure if necessary to have $\mu(K)=1$. There is clearly $C>0$ such that for $K_C:=\{ t\in K: \, \omega(t)>1/C\}$ we have $\mu(K_C)>\dfrac{3}{4}$. Increasing $L$ if necessary, we can suppose that $e\in L=L^{-1}$. We can also take $\e>0$ small enough to guarantee $4\e<\|\lambda \cchi_{K}\|_{p,\omega}$ and $2C\e<|\lambda|$.

Suppose the contrary, that is, $\|\lambda \cchi_{K}-\gamma T_sf\|_{p,\omega}\ge\e$ for every $\gamma\in \Gamma$, $s\in S\setminus L$. Let us prove the following

\textbf{Claim.} There exists $M>0$  such that for all $\gamma\in\Gamma$ and $s\in S$ satisfying $\| \lambda \cchi_{K}-\gamma T_sf\|_{p,\omega}<\e$, we have $s\in L$ and $|\gamma|>\xi$ with $\xi:=M^{-1}(|\lambda|-2^{1/p}C\e)>0$.
\begin{proof}[Proof of the claim]
There exists $M>0$ such that the set $X_M=\{ t\in LK:\, |f(t)|>M\}$ has measure less than $\dfrac{1}{10}\mu(K_C)$. Let $\gamma\in\Gamma$ and $s\in S$ be such that $\| \lambda \cchi_{K}-\gamma T_sf\|_{p,\omega}<\e$. By assumption we have $s\in L$ and
\begin{align*}
\e^p&>\|\lambda \cchi_{K}-\gamma T_sf\|_{p,\omega}^p\ge\int_{K_C}|\lambda-\gamma f(s^{-1}t)|^p\omega(t)^p\,d\mu(t)
\\
&\ge C^{-p}\int_{K_C}|\lambda-\gamma f(s^{-1}t)|^p\,d\mu(t)\ge C^{-p} \int_{(s^{-1}K_C)\setminus X_M}|\lambda-\gamma f(t)|^p\,d\mu(t).
\end{align*}
If $|\gamma|M<|\lambda|$, then
\begin{align*}
\e^p& \ge C^{-p} \int_{(s^{-1}K_C)\setminus X_M}(|\lambda|-|\gamma| M)^p\,\mathrm{d}\mu(t)\\
&> C^{-p}(|\lambda|-|\gamma| M)^p(\mu(K_C)-\mu(X_M))\\
&>\dfrac{9\mu(K_C)}{10C^p}(|\lambda|-|\gamma| M)^p>\dfrac{1}{2C^p}(|\lambda|-|\gamma| M)^p,
\end{align*}
which implies
\[|\gamma|>M^{-1}(|\lambda|-2^{1/p}C\e)=:\xi>0\]
If $|\gamma|M\ge|\lambda|$ then also $|\gamma|>\xi$, since $\xi<\dfrac{|\lambda|}{M}$.
\end{proof}
The set $LK$ is compact, so its measure is finite; pick $N\in\mathbb{N}$ with $\mu(LK)<N/2$. Set $r:=\max\{ 1, \|\cchi_{K}\|_{p,\omega}\}$ and $\lambda_j=\lambda+j\e/2Nr$, $j=0,...,N-1$.

For $j=0,...,N-1$ we choose $\delta_j>0$, $\gamma_j\in\Gamma$, $s_j\in S$ such that
$$
\|\lambda_j \cchi_{K}-\gamma_j T_{s_j}f\|_{p,\omega}<\delta_j.
$$
The choice of $\delta_j$ is specified later, but it will be less than $\e/4$. Since
$$
\|\lambda \cchi_{K}-\gamma_j T_{s_j}f\|_{p,\omega}<\delta_j+|\lambda-\lambda_j|\|\cchi_{K}\|_{p,\omega}<\delta_j + \frac\e{2}<\e,
$$
by the claim we have $s_j\in L$ and $|\gamma_j|>\xi$. Moreover, we have
\begin{align*}
\int_{K_C}|\lambda_j-\gamma_jf(s_{j}^{-1}t)|^p\,\mathrm{d}\mu(t)&<C^p\int_{K_C} \omega(t)^p|\lambda_j-\gamma_jf(s_{j}^{-1}t)|^p\,\mathrm{d}\mu(t)\\
&\le C^p\|\lambda_j \cchi_{K}-\gamma_j T_{s_j}f\|_{p,\omega}^p<C^p\delta_{j}^p.
\end{align*}
Set $K_{j}:=\{ t\in K_C:\, |\lambda_j-\gamma_j f(s_{j}^{-1}t)|\ge 4^{1/p}\delta_j C\}$. By trivial estimates,
\[C^p\delta_{j}^p>\int_{K_j}|\lambda_j-\gamma_j f(s_{j}^{-1}t)|^p\,\mathrm{d}\mu(t)\ge 4C^p\delta_{j}^p\mu(K_j),\]
so that
\[\mu(K_j)<\dfrac{1}{4}.\]
For $t\in K_{j}':=s_{j}^{-1}(K_C\setminus K_j)$ we have the estimate
\[|f(t)-\dfrac{\lambda_j}{\gamma_j}|<4^{1/p}C\dfrac{\delta_j}{|\gamma_j|}<4^{1/p}C\dfrac{\delta_j}{\xi}\]
and $\mu(K_{j}')\ge \mu(K_C)-\mu(K_j)>\dfrac{1}{2}$. The aim is to choose $\delta_j$ so that these sets are pairwise disjoint. For this, it is sufficient to choose $\delta_j$ so that the disks $B_j:=\{ z\in\mathbb{C}: |z-\lambda_j/\gamma_j|<4^{1/p}C\delta_j/\xi\}$ are pairwise disjoint. Suppose by induction that this is done for $k<j$. Set $\delta_j=\min\{ \delta_k:\, k<j\}$. If $B_j\cap(\cup_{k<j}B_k)=\emptyset$, then we are done. If not, divide $\delta_j$ by $2$ and  pick $\gamma_j\in\Gamma$, $s_j\in S$ accordingly. If we still have $B_j\cap(\cup_{k<j}B_k)\neq\emptyset$, continue in the same way. Either we arrive at an empty intersection on some step --- in this case we can pass to the next $j$ --- or we continue infinitely and get a sequence $\gamma_{j,n}\in\Gamma$, $s_{j,n}\in S$ such that $\|\lambda_j \cchi_{K}-\gamma_{j,n}T_{s_{j,n}}f\|_{p,\omega}\to0$ as $n\to\infty$. According to the Claim, we have $s_{j,n}\in L$ and $|\gamma_{j,n}|>\xi$, so that
\begin{align*}
\int_{s_{j,n}^{-1}(G\setminus K)}|f(t)|^p\omega(s_{j,n}t)^p\,\mathrm{d}\mu(t)&=\int_{G\setminus K}|(T_{s_{j,n}}f)(t)|^p\omega(t)^p\,\mathrm{d}\mu(t)\\
&\le \|\dfrac{\lambda_j}{\gamma_{j,n}}\cchi_{K}-(T_{s_{j,n}}f)(t)\|_{p,\omega}^p\\
&\le \dfrac{1}{\xi}\|\lambda_j \cchi_{K}-\gamma_{j,n}T_{s_{j,n}}f\|_{p,\omega}\to0.
\end{align*}
Since $G\setminus LK\subset s_{j,n}^{-1}(G\setminus K)$, we have
\begin{equation}\label{int-fw}
\int_{G\setminus LK} |f(t)|^p\omega(s_{j,n}t)^p\,\mathrm{d}\mu(t)\to0 \text{ as } n\to\infty.
\end{equation}
This is an exercise to show that \eqref{int-fw} implies $f\equiv0$ on $G\setminus LK$. [
If not, there is a compact set $F\subset G\setminus LK$ of positive measure and $\delta>0$ such that $|f|\ge\delta$ on $F$. Recall that $s_{j,n}\in L$ for every $j,n$. Pick $\delta_1>0$ such that the set $D=\{t\in LF: \omega(t)\le \delta_1\}$ has measure $\mu(D)<\frac1{10}\mu(F)$. Now
\begin{align*}
\int_F |f(t)|^p \omega(s_{j,n}t)^pd\mu(t) &\ge \int_{F\setminus s_{j,n}^{-1}D} \delta^p\delta_1^p \,\mathrm{d}\mu(t)
\\&\ge \delta^p\delta_1^p \big( \mu(F) - \mu(D) \big) > \frac9{10}\delta^p\delta_1^p \mu(F)
\end{align*}
and cannot tend to 0.]

But now it is easy to show that $f$ cannot be $(\Gamma,S)$-dense. Indeed, let $g\in G$ be such that $LKK^{-1}g\cap LKK^{-1}=\emptyset$. For every $s\in G$ we have then either $s^{-1}K\cap LK=\emptyset$ or $s^{-1}gK\cap LK=\emptyset$. In the first case, $T_sf$ vanishes on $K$, and in the second case, on $gK$. For all $\gamma\in \Gamma$ and $s\in S$, we have then 
\begin{align*}
\|\cchi_{K}+\cchi_{gK} - \gamma T_s f\|_{p,\omega}^p 
&\ge \int_K |1 - \gamma (T_s f)(t)|^p\omega^p(t)\,\mathrm{d}\mu(t) + \int_{gK} |1-\gamma (T_s f)(t)|^p \omega^p(t)\,\mathrm{d}\mu(t)
\\&\ge \min \Big\{ \int_K \omega^p(t)\,\mathrm{d}\mu(t),  \int_{gK} \omega^p(t)\,\mathrm{d}\mu(t) \Big\}
\end{align*}
so that $\gamma T_sf$ cannot approximate $\cchi_{K}+\cchi_{gK}$ arbitrarily well.

This shows that the choice of required $\delta_j$ is always possible, yielding the pairwise disjoint sets $K_{j}'$. It follows that $\mu(\cup K_{j}')>\dfrac{N}{2}>\mu(LK)$, in particular, $\cup K_{j}'\nsubseteq LK$. But by assumption, we have $K_{j}'\subset s_{j}^{-1}K\subset LK$ for every $j$. This contradiction proves the lemma.
\end{proof}

\begin{rem}
This lemma implies that in case of existence of $(\Gamma,S)$-dense vectors, not only $G$ cannot be compact, but also $S$ cannot be contained in a compact subset of a non-compact group $G$.
\end{rem}

\section{Proof of Theorem \ref{theorem A}}\label{sec-proof-A}

From now on, we assume the following.
\begin{defn}\label{assump}
We say that $(G,S,p,\omega,\Gamma)$ is an {\it admissible tuple} if $G$ is a second-countable locally compact non-compact group,
$S\subset G$, $1\le p<+\infty$, $\Gamma\subset\mathbb{C}$ is such that $\Gamma\setminus\{0\}$ is non-empty,
and $\omega$ is a locally $p$-integrable $S$-admissible weight on $G$.
\end{defn}

\begin{Theorem A}
If $(G,S,p,\omega,\Gamma)$ is admissible, the following conditions are equivalent:
\begin{enumerate}[label={$(\arabic*)$}]
\item There is a $(\Gamma,S)$-dense vector in $L^p(G,\omega)$.
\item For every increasing sequence of compact subsets  $( F_n)_{n\ge1}$ of $G$ of positive measure and for every sequence  $(\delta_n)_{n\ge1}$ of positive numbers, there are sequences $(s_n)_{n\ge1}\subset S
$, $(\lambda_n)_{n\ge1}\subset\Gamma\setminus\{0\}$ and a sequence of compact subsets $K_n\subset F_n$ such that the sets $s_{n}^{-1}F_n$ are pairwise disjoint, $\mu(F_n\setminus K_n)<\delta_n$ and
\[
\sum_{n,k\ge0;n\neq k}\dfrac{|\lambda_n|^p}{|\lambda_k|^p}\|\omega\|_{p,s_ns_{k}^{-1}K_{k}}^p<+\infty,
\]
with  $s_0=e$, $\lambda_0=1$ and $K_0=\emptyset$.
\item For every $N\ge1$, there exist $N$ vectors $\{ f_1,...,f_N\}\subset L^p(G,\omega)$ such that the set
\[\{ \lambda( T_sf_1,..., T_sf_N):\, \lambda\in \Gamma, \, s\in S\}\]
is dense in the direct sum of $N$ copies of $L^p(G,\omega)$.
\end{enumerate}
\label{1}
\end{Theorem A}

\begin{proof}[Proof of Theorem A]
Firstly, it is clear that $(3)\Rightarrow (1)$. Let us show that $(1)\Rightarrow(2)$. Assume that $s_0=e$, $\lambda_0=1$ and $K_0=\emptyset$. We can suppose that the sequence $(\delta_n)_{n\ge1}$ is decreasing with $\sum_{k\ge1}\delta_k<+\infty$ and $\delta_k<\mu(F_k)$. For each $k\ge1$, there exist a compact set
 $F_{k}'\subset F_{k}$ and $0<c_k<\frac{1}{4}$ such that $\mu(F_k\setminus F_{k}')<\dfrac{\delta_k}{2}$ and
\begin{equation}
 \underset{t\in F_{k}'}{\mathrm{ess }\inf}\,\omega(t)\ge c_{k}\quad,\quad\underset{t\in F_{k}'}{\mathrm{ess }\sup}\,\omega(t)\le c_{k}^{-1}.
 \label{11}
\end{equation}
Let $f\in L^p(G,\omega)$ be a $(\Gamma,S)$-dense vector. According to Lemma \ref{9}, there exist two sequences $(s_k)_{k\ge1}\subset S$ and $(\lambda_k)_{k\ge1}\subset\Gamma\setminus\{0\}$ such that, for each $k\ge1$, we have
 \begin{equation}
 \| 2 \cchi_{F_{k}'}-\lambda_k T_{s_k}f\|_{p,\omega}^p<c_{k}^{2p}\delta_k,
 \label{10}
 \end{equation}
and $s_{k}^{-1}F_k\cap s_{j}^{-1}F_j=\emptyset$ for every $1\le j< k$. By \eqref{11} and \eqref{10}, we get 
\[
c_{k}^{2p}\delta_{k}>\int_{F_{k}'}\vert \lambda_k(T_{s_k}f)(t)-2\vert^p\omega(t)^p\,\mathrm{d}\mu(t)
\ge c_{k}^p\int_{F'_{k}}\vert \lambda_k(T_{s_k}f)(t)-2\vert^p\,\mathrm{d}\mu(t),
\]
that is \[\Vert\lambda_k T_{s_k}f-2\cchi_{F_{k}'}\Vert^p_{p,F'_{k}}<c_{k}^p\delta_k<\dfrac{\delta_k}{4^p}<\delta_k.\] 
Set $K_{k}':=\{ t\in F_{k}':\, |\lambda_k f(s_{k}^{-1}t)|>1\}$. It follows that  
\begin{equation*}
\dfrac{\delta_k}{4}>c_{k}^p\delta_k>\int_{F_{k}'\setminus K_{k}'}|2-\lambda_k f(s_{k}^{-1}t)|^p\,\mathrm{d}\mu(t)\ge\mu(F_{k}'\setminus K_{k}'),
\end{equation*}
There exists a compact subset
$K_k\subset K_{k}'$ such that $\mu(K_{k}'\setminus K_{k})<\dfrac{\delta_k}4$. According to the above estimates, we get $\mu(F_k\setminus K_k)<\delta_k$. Moreover,
\begin{equation}
\sum_{k>0}\dfrac{1}{|\lambda_k|^p}\|\omega\|_{p,s_{k}^{-1}K_{k}}^p
\le\sum_{k}\int_{s_{k}^{-1}K_{k}}\omega(t)^p|f(t)|^p\,d\mu(t)
\le\| f\|_{p,\omega}^p<+\infty.
\label{12}
\end{equation}
Fix $n\ge1$. Note that, for every $k\neq n$ and $t\in s_ns_{k}^{-1}K_k$, we have  $t\in G\setminus K_{n}'$ and $|\lambda_k(T_{s_n}f)(t)|>1$. Moreover, the sets $s_ns_{k}^{-1}K_k$ are pairwise disjoint. Hence
\begin{align}
\sum_{k,k\neq n}\dfrac{|\lambda_n|^p}{|\lambda_k|^p}\|\omega\|^p_{p,s_{n}s_{k}^{-1}K_k}
&\le\sum_{k,k\neq n}\int_{s_{n}s_{k}^{-1}K_k}\omega(t)^p| \lambda_n(T_{s_n}f)(t)|^p\,\mathrm{d}\mu(t) \nonumber\\
&\le \int_{G\setminus K_{n}'}\omega(t)^p|\lambda_n (T_{s_n}f)(t)|^p\,\mathrm{d}\mu(t) \nonumber\\
&\le c_{n}^{-p}\int_{F_{n}'\setminus K_{n}'}|\lambda_n (T_{s_n}f)(t)|^p\,\mathrm{d}\mu(t) \nonumber
\\&\quad + \int_{G\setminus F_n'} \omega(t)^p| 2\cchi_{F_n'} - \lambda_n (T_{s_n}f)(t)|^p\,\mathrm{d}\mu(t) \nonumber
\\&\le c_{n}^{-p}\mu(F_{n}'\setminus K_{n}')+\|2\cchi_{F_{n}'}-\lambda_n T_{s_n}f\|_{p,\omega}^p\nonumber\\
&<\delta_n+ c_{n}^{2p}\delta_n<2\delta_n.
\label{13}
\end{align}
According to \eqref{12} and \eqref{13} and since $\sum_{n\ge1}\delta_n<+\infty$, we obtain that the series in the condition $(2)$ of Theorem \ref{theorem A} converges.

$(2)\Rightarrow (3)$. Without  loss of generality we can assume that $N=2$. Since $L^p(G,\omega)$ is separable, there exists a countable sequence  $\{ (p_n,q_n):\,n\ge1\}\subset \mathcal{K}(G)\times \mathcal{K}(G)$ dense in $L^p(G,\omega)\times L^p(G,\omega)$ (where $\mathcal{K}(G)$ is the set of essentially bounded functions on $G$ with compact support). Let $\{ (g_n,h_n):\, n\ge1\}$ be a sequence composed by the terms $(p_n,q_n)$ so that each term appears infinitely many times. 

Set $F_n:=\cup_{k=1}^{n} S_k$ where $S_k:=\mathrm{supp}\,g_k\cup\mathrm{supp}\,h_k$ and $\alpha_n:=\max\left\{ \| g_n\|_{\infty}^p,\| h_n\|_{\infty}^p\right\}$. Let also $(\delta_n')$ be a decreasing sequence such that $0<\delta_{n}'<\alpha_{n}^{-1}2^{-n}$. Since $\|\omega\|_{p,F_n}^p<+\infty$ for every $n\ge1$, there exists a positive decreasing sequence $(r_n)_{n\ge1}$ such that $\|\omega\|_{p,E}^p<\delta_{n}'$ for every $E\subset F_n$ satisfying $\mu(E)<r_n$. Set $\delta_n=\min(r_n ;\delta_{n}')>0$. 

By assumption, there exist $(s_n)_{n\ge1}\subset S$, $(\lambda_n)_{n\ge1}\subset\Gamma\setminus\{0\}$, and $K_n\subset F_n$ such that the sets $s_{n}^{-1}F_n$ are pairwise disjoint,  $\mu(F_n\setminus K_n)<\delta_n$ and
\[\sum_{k\ge0}a_k<+\infty\quad\text{where } a_k=\sum_{n;n\neq k}\dfrac{|\lambda_n|^p}{|\lambda_k|^p}\|\omega\|_{p,s_ns_{k}^{-1}K_k}^p\]
with $s_0=e$, $\lambda_0=1$ and $K_0=\emptyset$. Hence $a_k\underset{k\to+\infty}{\longrightarrow}0$, thus there exists a subsequence $(a_{l_k})_{ k\ge0}$ of $(a_k)_{k\ge0}$ (with $l_0=0$, $l_k\ge k$) such that $\sum_{k\ge1} \alpha_ka_{l_k}<+\infty$. We can also choose $l_k$ so that the pair $(g_{l_k},h_{l_k})$ is the same as $(g_k,h_k)$ for every $k\ge1$. We then have
\begin{equation}
\sum_{n,k;n\neq k}\alpha_k\dfrac{|\mu_n|^p}{|\mu_k|^p}\Vert \omega\Vert_{p,t_nt_{k}^{-1}E_k}^p<+\infty
\label{14}
\end{equation}
where $t_n=s_{l_n}$, $\mu_n=\lambda_{l_n}$ and $E_k=K_{l_k}$. Moreover, the sets $t_{n}^{-1}E_n$ are pairwise disjoint, and for every $n\ge1$, $\mu(S_n\setminus E_n)\le  \mu(F_{l_n}\setminus K_{l_n})<\delta_{l_n}\le\delta_n$. So, the series
\[f_1:=\sum_{k\ge0}\dfrac{1}{\mu_k}T_{t_{k}^{-1}}(g_k\cchi_{E_k})\quad\text{ and }\quad f_2=:\sum_{k\ge0}\dfrac{1}{\mu_k}T_{t_{k}^{-1}}(h_k\cchi_{E_k}),\]
are convergent. Indeed, according to \eqref{14}, we get 
\[
\|f_1\|_{p,\omega}^p
\le \sum_k \frac1{|\mu_k|^p} \int_{t_k^{-1}E_k} \omega(t)^p |g_k(t_kt)|^p dt
\le \sum_{k} \alpha_{k}\dfrac{1}{|\mu_k|^p}\| \omega\|_{p,t_{k}^{-1}E_k}^p<+\infty.
\]
Similarly, one can check that $\|f_2\|_{p,\omega}^p<\infty$.
To finish the proof, we have to show that  $\Vert \mu_n(T_{t_n}f_1,T_{t_n}f_2)-(g_n,h_n)\Vert\underset{n\to+\infty}{\longrightarrow}0$. Note that, for any $n\ge1$ and $i\neq j$ such that $i,j\neq n$, we have   $t_nt_{i}^{-1}E_{i}\cap t_nt_{j}^{-1}E_{j}=\emptyset$ and $E_n\cap t_nt_{i}^{-1}E_i=\emptyset$. Now, for $n\ge1$, we have 
\begin{align*}
\Vert\mu_n T_{t_n}f_1-g_n\Vert_{p,\omega}^p&=\underset{k,k\neq n}{\sum}\dfrac{|\mu_n|^p}{|\mu_k|^p}\Vert T_{t_nt_{k}^{-1}}(g_k\cchi_{E_k})\Vert_{p,\omega}^p+\Vert g_n(\cchi_{E_n}-1)\Vert_{p,\omega}^p \nonumber\\
&\le\underset{k,k\neq n}{\sum}\dfrac{|\mu_n|^p}{|\mu_k|^p}\Vert g_k\Vert_{\infty}^p\|\omega\|_{p,t_nt_{k}^{-1}E_k}^p+\Vert g_n\Vert^p_{\infty} \|\omega\|_{p,\mathrm{supp} (g_n)\setminus E_n}^p \nonumber\\
&\le \underset{k,k\neq n}{\sum}\dfrac{|\mu_n|^p}{|\mu_k|^p}\alpha_k\|\omega\|_{p,t_nt_{k}^{-1}E_k}^p+\alpha_n\delta_{l_n}' \nonumber\\
&\le \displaystyle\e_n+ 2^{-n},
\end{align*}
where
\[\displaystyle\e_n:=\sum_{k;k\neq n}\frac{|\mu_n|^p}{|\mu_k|^p} \alpha_k \|\omega\|_{p,t_nt_{k}^{-1}E_k}^p.\]
Similarly, $
\Vert\mu_n T_{t_n}f_2-h_n\Vert_{p,\omega}^p\le\e_n+2^{-n}
$. By
 \eqref{14}, we get  $\e_n\to0$. Hence,
\[
\|\mu_nT_{t_n}f_1-g_n\|_{p,\omega} + \|\mu_nT_{t_n}f_2-h_n\|_{p,\omega}
\le 2(2^{-n}+\e_n)^{1/p}\underset{n\to+\infty}{\longrightarrow}0.
\]
\end{proof}

\begin{rem}
If $G$ is a countable discrete group, then in condition $(2)$ of Theorem \ref{theorem A} we can suppose that the set $K_n$ is equal to $F_n$.
\end{rem}

The $(\Gamma,S)$-density condition, for $\Gamma$ unbounded or unbounded away from zero, is strictly weaker than just $S$-density, as can be seen on the following example.

\begin{ex}\label{no-S-dense}
Let $G=\mathbb{Z}$, let $S=( n_k)_{n\ge1}$ be a strictly increasing sequence of positive integers, let $1\le p<+\infty$ and let $\Gamma\subset\mathbb{C}$ be such that $\Gamma\setminus\{0\}$ is non-empty. Suppose that one of the following conditions holds:
\begin{enumerate}[label={$(\arabic*)$}]
\item $\Gamma$ is unbounded and \[ \omega(n)=\begin{cases}
 2^{-n} &\text{ if } n\ge0\\
 1  &\text{ if } n\le0
\end{cases}.
\]
\item $0$ is a limit point of $\Gamma$ and \[ \omega(n)=\begin{cases}
 1 &\text{ if } n\ge0\\
 2^{n} &\text{ if } n\le0
\end{cases}.
\]
\end{enumerate}
Then $\omega$ is a $\mathbb{Z}$-admissible weight such that $L^p(\mathbb{Z},\omega)$ has a $(\Gamma,S)$-dense vector, but has no $S$-dense vectors.
\label{15}
\end{ex}
\begin{proof}We prove the case $(1)$ only, and the case $(2)$ is proved similarly.
It is clear that $\omega$ is $\Z$-admissible.  Since $\underset{n\in\mathbb{N}}{\inf}\,\max(\omega_n,\omega_{-n})=1\neq0$, it follows from \cite[Corollary 13]{AK} that $L^p(\Z,\omega)$ has no $S$-dense vector. Let us show that the condition $(2)$ of Theorem \ref{theorem A} holds.
Let $( F_n)_{n\ge1}$ be an increasing sequence of finite sets of $\Z$. For every $n\in\N$, set $D_n:=\max\{ t: t\in F_n\}$, $d_n:=\min\{ t:\, t\in F_n\}$, $s_0=0$, $\lambda_0=1$ and $F_0=\emptyset$. We denote by $\mathrm{card}(F)$ the cardinality of a set $F$. By induction, we construct a subsequence $(s_n)_{n\ge1}\subset S$ and a sequence $(\lambda_n)_{n\ge1}\subset\Gamma\setminus\{0\}$ such that for every $k=0,1,..., n-1$,
\begin{enumerate}[label={$(\arabic*)$}]
\item $\dfrac{|\lambda_k|^p}{|\lambda_n|^p}\,\mathrm{card}(F_n)<\dfrac1{2^n}$ (this choice  of $\lambda_n$ is possible since $\Gamma$ is unbounded);
\item $D_n +s_k-d_k<s_n$;
\item $\dfrac{|\lambda_n|^p}{|\lambda_k|^p}\,\mathrm{card}(F_n)\, \omega(d_n+s_n-s_k)^p <\dfrac{1}{2^n}$ (when choosing $s_n$, both (2) and (3) can be satisfied since $\omega(n)\to0$ as $n\to+\infty$).
\end{enumerate}
It is clear that condition $(2)$ implies that $(F_k-s_k)\cap (F_n-s_n)=\emptyset$ for every $k<n$. Since $\omega$ is decreasing and bounded by 1,
 $F_k\subset F_n$ for every $k<n$, 
we have
\begin{align*}
\sum_{n,k\ge0; n\neq k}\dfrac{|\lambda_n|^p}{|\lambda_k|^p}\|\omega\|_{p,F_{k}+s_n-s_k}^p
&=\sum_{n\ge0} \sum_{k\ne n}\dfrac{|\lambda_n|^p}{|\lambda_k|^p}\sum_{t\in F_k}\omega(t+s_n-s_k)^p
\\ &\le\sum_{n\ge0}\Bigg( \sum_{k<n}\dfrac{|\lambda_n|^p}{|\lambda_k|^p} \sum_{t\in F_n}\omega(t+s_n-s_k)^p+\sum_{k>n}\dfrac{|\lambda_n|^p}{|\lambda_k|^p}\,\mathrm{card}(F_k)\Bigg)\\
&\le\sum_{n\ge0}\left( \sum_{k<n}\dfrac{|\lambda_n|^p}{|\lambda_k|^p}\,\mathrm{card}(F_n)\,\omega(d_n+s_n-s_k)^p+\sum_{k>n}\dfrac{|\lambda_n|^p}{|\lambda_k|^p}\,\mathrm{card}(F_k)\right)\\
&<\sum_{n\ge0}\Big( \sum_{0\le k< n}\dfrac{1}{2^n}+\sum_{k>n}\dfrac{1}{2^k}\Big)\quad (\text{by } (1) \text{ and } (3)) \\
&=\sum_{n\ge0}\Big(\dfrac{n}{2^n}+\dfrac{1}{2^n}\Big)<+\infty.
\end{align*}
This proves that $L^p(\mathbb{Z},\omega)$ has a $(\Gamma,S)$-dense vector.
\end{proof}

\section{Sufficient conditions}

The condition of Theorem \ref{theorem A} is universal, but sometimes difficult to check. It is therefore convenient to have simpler sufficient conditions, even if they are more restrictive.
We make on $G,S,p,\Gamma,\omega$ the same assumptions as in Section 4, see Definition \ref{assump}.

\begin{thm}\label{suff-sup}
Suppose that $(G,S,p,\omega,\Gamma)$ is an admissible tuple. If for every compact subset $K$ of $G$ and for all $\e,\delta>0$, there exist $s\in S$ , $\lambda\in \Gamma\setminus\{0\}$ and a compact subset $E\subset K $ such that $\mu(K\setminus E)<\delta$ and 
\[
|\lambda|\,\underset{sE}{\mathrm{ess }\sup }\,\omega<\e \quad\text{ and }\quad 
\frac1{|\lambda|}\,\underset{s^{-1}E}{\mathrm{ess }\sup}\,\omega<\e,
\]
then $L^p(G,\omega)$ has a $(\Gamma,S)$-dense vector.
\label{18}
\end{thm}

We could of course state the theorem just with one parameter $\e$, implying $\delta=\e$, and get an equivalent statement.
For the proof, we need the following lemma.
\begin{lem}
Fix a compact set $L\subset G$. Under the assumptions of Theorem \ref{suff-sup}, $s$ can in addition be chosen outside of $L$.
\label{17}
\end{lem}
\begin{proof}
Let $K,L$ be two compact subsets of $G$. Assume, towards a contradiction, that for some $\e,\delta>0$ the assumptions of Theorem \ref{suff-sup} are only satisfied if $s\in S\cap L$. This implies in particular that $\mu(K)>0$, otherwise $E=\emptyset$ and any $s\in S\setminus L$ will do  (by Lemma \ref{noncompact}, $S$ cannot be contained in $L$). We can clearly decrease $\delta$ as to have $0<\delta<\mu(K)$.

Let $E_0\subset K$, $s_0\in S\cap L$ and $\lambda_0\in\Gamma\setminus \{0\}$ satisfy the assumptions of the theorem with $\e,\delta>0$. 
Set 
\[t_0:=
    \begin{cases}
    s_0 &\text{ if }|\lambda_0|\ge1 \\
    s_{0}^{-1} &\text{ if }|\lambda_0|<1
    \end{cases}.
    \]
Then
\[
\underset{t_0E_0}{\mathrm{ess }\sup }\,\omega<\e
\]
and $\mu(E_0)>\mu(K)-\delta=\delta_1>0$. There is next a compact subset $K_1\subset E_0$ such that $\mu(E_0\setminus K_1)<\dfrac{\delta_1}{9}$ and  $\e_1:=\underset{t_0K_1}{\mathrm{ess }\inf}\,\omega>0$.

By induction, we can choose for $n\ge0$ sequences $E_n$, $K_n$ of compact sets, $s_n\in S\cap L$, $\lambda_n\in\Gamma\setminus\{ 0\}$, $\e_n>0$ (with $K_0:=K$ and $\e_0:=\e$) as follows. For a given $n\ge1$, one chooses first $E_n$, $\lambda_n$, $s_n$ such that
$$
E_n\subset K_n, \qquad \mu(K_n\setminus E_n)<\dfrac{\delta_1}{9^n},
\qquad
|\lambda_n|\,\underset{s_nE_n}{\mathrm{ess }\sup }\,\omega<\e_n, \qquad 
\frac1{|\lambda_n|}\,\underset{s_n^{-1}E_n}{\mathrm{ess }\sup}\,\omega<\e_n.
$$
Setting \[t_n:=
    \begin{cases}
    s_n &\text{ if }|\lambda_n|\ge1 \\
    s_{n}^{-1} &\text{ if }|\lambda_n|<1
    \end{cases},\]
 we have then
\begin{equation}
\underset{t_nE_n}{\mathrm{ess }\sup }\,\omega<\e_n.
\label{27}
\end{equation}
Now we choose $K_{n+1}$ so that
\begin{equation}
K_{n+1}\subset E_n, \qquad \mu(E_n\setminus K_{n+1})<\frac{\delta_1}{9^{n+1}}, \qquad
  \e_{n+1}:=\underset{t_n K_{n+1}}{\mathrm{ess }\inf}\,\omega>0.
  \label{28}
\end{equation}
Setting $D_n=t_nE_n$, as in the proof of \cite[Theorem 8]{AK},  we first note that
$$
\mu(D_n) = \mu(E_n) > \mu(K_n)-\frac{\delta_1}{9^n} > \mu(E_{n-1})-\frac{2\delta_1}{9^n},
$$
so that by induction
$$
\mu(D_n) > \mu(E_0) - \sum_{k=1}^n \frac{2\delta_1}{9^k} >\frac34\,\delta_1.
$$
Next, the estimates \eqref{27} and \eqref{28} imply that $\mu(t_n K_{n+1} \cap t_k E_k)=0$ for any $k>n$, so that up to a null set, 
$$
D_n \cap D_k = t_n E_n \cap t_k E_k\subset t_nE_n\setminus (t_n K_{n+1}),
$$
and
$$
\mu\Big( \bigcup_{k>n} (D_n \cap D_k) \Big) \le \mu \big( t_nE_n\setminus (t_n K_{n+1}) \big) 
< \frac{\delta_1}{9^{n+1}}.
$$
It follows that
$$
\mu\Big( \bigcup_{k<n} (D_n \cap D_k) \Big) \le \sum_{k<n} \frac{\delta_1}{9^{k+1}} < \frac{\delta_1}{8},
$$
and as a consequence
\begin{align*}
\mu\Big(\bigcup_{n\le N}D_n\Big) &\ge \sum_{n=1}^N \mu\Big(D_n\setminus \bigcup_{k< n}D_k\Big)
\ge \sum_{n=1}^N \Big[ \mu(D_n) - \mu\Big(\bigcup_{k<n} (D_n \cap D_k) \Big) \Big]
\\& > \sum_{n=1}^N [ \frac34\,\delta_1 - \frac18\,\delta_1 ] = \frac{5N}8\, \delta_1.
\end{align*}
But now $\mu(LE_0\cup L^{-1}E_0)\ge\mu(\cup_{k=1}^{\infty}D_k)=+\infty$, which is impossible since $LE_0\cup L^{-1}E_0$ is compact.
\end{proof}

\begin{proof}[Proof of Theorem \ref{suff-sup}]
 For $K\subset G$, denote $\underset{K}{\mathrm{ess }\sup }\,\omega$ by $\|\omega\|_{\infty,K}$. Let us show that condition $(2)$ of Theorem \ref{theorem A} holds. Let $(F_n)_{n\ge1}$ be an increasing sequence of compact sets of positive measure and let $(\delta_n)_{n\ge1}$ be a sequence of positive numbers. We can assume that $\delta_1<\dfrac12$, $\delta_n<\mu(F_n)$ and $\delta_{n+1}<\dfrac{\delta_n}{2}$ for every $n\ge1$. Set $F_{0}=\emptyset$, $s_0=e$ and $\lambda_0=1$. By induction, we choose sequences $(s_n)_{n\ge1}\subset S$ and $(\lambda_n)_{n\ge1}\subset\Gamma\setminus\{0\}$ as follows. If $s_k$, $\lambda_k$ are chosen for $k=0,\dots,n-1$, set
$$
E_n:=(\bigcup_{k<n}s_kF_n)\cup(\bigcup_{k<n}F_k\cup s_{k}^{-1}F_k)\quad\text{ and }
\quad C_n:=\max_{k<n}\| T_{s_k}\|\ge1.
$$
Choose $\e_n>0$ so that $\e_n^p \,\mu(E_n) \max(n,C_n^p) < 2^{-n}$. There exist a compact set $E_{n}'\subset E_n$ and $s_n\in S$, $\lambda_n\in\Gamma$ such that $\mu(E_n\setminus E_{n}')<\dfrac{\delta_n}{2}$,
\[
|\lambda_n|\|\omega\|_{\infty,s_nE_{n}'}<\e_n\min_{k<n}|\lambda_k|\quad\text{ and }\quad \dfrac{1}{|\lambda_n|}\|\omega\|_{\infty,s_{n}^{-1}E_{n}'}<\dfrac{\e_n}{\max_{k<n}|\lambda_k|}.
\]
Moreover, by Lemma \ref{17} we can choose $s_n$ such that $s_{n}^{-1}F_n\cap s_{k}^{-1}F_k=\emptyset$ for every $k<n$. Set now
$$
K_0=\emptyset\quad\text{ and }\quad K_n=F_n\cap E_{n}'\cap s_n(\cap_{k>n}E_{k}')\quad \text{ for every } n\ge1.
$$
For fixed $n\ge1$, we have $F_n\subset E_n$ and  for every $k>n$, $F_n\subset s_nE_k$, thus
\begin{align*}
\mu(F_n\setminus K_n)&\le \mu(E_n\setminus E_{n}')+\mu(\cup_{k>n}s_n E_k\setminus s_nE_{k}')\\
&\le \sum_{k\ge n}\mu(E_k\setminus E_{k}')<\sum_{k\ge n}\dfrac{\delta_k}{2}<\sum_{k\ge n}\dfrac{\delta_n}{2^{k-n+1}}=\delta_n.
\end{align*}
Since the sets $s_{n}^{-1}F_n$ are pairwise  disjoint and $s_{k}^{-1}K_k\subset E_{n}'$ for every  $k<n$, we have
\begin{align*}
\sum_{n,k\ge0;n\neq k}\dfrac{|\lambda_n|^p}{|\lambda_k|^p}\|\omega\|_{p,s_ns_{k}^{-1}K_{k}}^p
&\le \sum_n \Big( \sum_{k<n} \dfrac{|\lambda_n|^p}{|\lambda_k|^p}\|\omega\|_{p,s_ns_{k}^{-1}K_{k}}^p
 + \sum_{k>n} \|T_{s_n}\|^p \dfrac{|\lambda_n|^p}{|\lambda_k|^p}\|\omega\|_{p,s_{k}^{-1}K_{k}}^p\Big)\\
&\le \sum_n \Big(\dfrac{n \,|\lambda_n|^p}{\min_{k<n}|\lambda_k|^p}\|\omega\|_{p,s_nE_{n}'}^p
 +\sum_{k>n}\|T_{s_n}\|^p \dfrac{|\lambda_n|^p}{|\lambda_k|^p}\|\omega\|_{p,s_{k}^{-1}E_{k}'}^p\Big)\\
&< \sum_n \Big(\dfrac{n\,|\lambda_n|^p}{\min_{k<n}|\lambda_k|^p}\|\omega\|_{\infty,s_nE_{n}'}^p\mu(E_{n}')
 +\sum_{k>n}C_{k}^p\dfrac{|\lambda_n|^p}{|\lambda_k|^p}\|\omega\|_{\infty,s_{k}^{-1}E_{k}'}^p\mu(E_{k}')\Big)\\ 
&\le \sum_n \Big(n\mkern2mu \e_n^p \mkern2mu \mu(E_{n}')
 +\sum_{k>n}C_{k}^p\,\e_k^p \mkern2mu \mu(E_{k}')\Big)\\ 
&<\sum_n \Big(\dfrac{1}{2^{n}}+\sum_{k>n}\dfrac{1}{2^{k}}\Big)<+\infty.
\end{align*}
\end{proof}

\begin{cor}
Let $(G,S,p,\omega,\Gamma)$ be admissible. If for every compact subset $K$ of $G$ there exists $\lambda\in \Gamma\setminus\{0\}$ such that
\[\underset{s\in S}{\inf}\max\left\{ \vert\lambda\vert\,\underset{sK}{\mathrm{ess }\sup }\,\omega\,;\, \dfrac{1}{\vert\lambda\vert}\,\underset{s^{-1}K}{\mathrm{ess }\sup}\,\omega\right\} =0,\]
then $L^p(G,\omega)$ has a $(\Gamma,S)$-dense vector.
\label{21}
\end{cor}

If we know that all translations are bounded, on the left and on the right, then we get the following proposition. It implies  Proposition \ref{3} when $G$ is abelian.

\begin{prop}\label{MR}
Suppose that $(G,S,p,\omega,\Gamma)$ is an admissible tuple. Let $\omega$ be a continuous weight such that for every $s\in G$, we have
\[
M(s):=\underset{t\in G}{\sup}\,\dfrac{\omega(st)}{\omega(t)}<+\infty \quad\text{and}
\quad R(s):=\underset{t\in G}{\sup}\,\dfrac{\omega(ts)}{\omega(t)}<+\infty.
\]
Then the following conditions are equivalent:
\begin{enumerate}[label={$(\arabic*)$}]
\item  There exists a $(\Gamma,S)$-dense vector in $L^p(G,\omega)$.
\item For every compact set $K\subset G$, there exist sequences $(s_n)_{n\ge1}\subset S$ and $(\lambda_{n})_{n\ge1}\subset\Gamma\setminus\{0\}$ such that
\[\underset{n\to+\infty}{\lim} |\lambda_n|\,\underset{t\in K}{\sup}\,\omega(s_nt)=0 \quad\text{and}\quad \underset{n\to+\infty}{\lim}\,\frac{1}{|\lambda_n|}\,\underset{t\in K}{\sup}\,\omega(s_{n}^{-1}t)=0.\]
\item There exist sequences $(s_n)_{n\ge1}\subset S$ and $(\lambda_{n})_{n\ge1}\subset\Gamma\setminus\{0\}$ such that
\[\underset{n\to+\infty}{\lim}\,|\lambda_n|\,\omega(s_n)=0 \quad\text{and}\quad \underset{n\to+\infty}{\lim}\,\frac{1}{|\lambda_n|}\,\omega(s_{n}^{-1})=0.\]
\end{enumerate}
\label{25}
\end{prop}
\begin{proof}
Note that $M$ or $R$ being finite implies already that $\omega$ can be chosen continuous \cite[Theorem 2.7]{feicht}, so that we can assume it from the beginning.
It is clear that $(2)\Rightarrow(3)$. By Theorem \ref{18}, we have $(2)\Rightarrow(1)$. Let us show first that $(3)\Rightarrow(2)$. Fix a compact set $K\subset G$. By \cite[Proposition 1.16]{E}, we know that $M$ and $R$ are locally bounded. It is easy to see that for $s\in S$ and $\lambda\in\Gamma\setminus\{0\}$,
\[
\max\{ |\lambda| \sup_{t\in K}\,\omega(st)\,;\,\frac1{|\lambda|} \sup_{t\in K}\,\omega(s^{-1}t)\}
 \le \sup_{t\in K} R(t)\, \max\{ |\lambda|\omega(s)\,;\,\frac{1}{|\lambda|}\omega(s^{-1})\},\]
thus $(2)$ and $(3)$ are equivalent. Let us show now that $(1)\Rightarrow(3)$. Let $F\subset G$ be a compact set of  nonzero measure. By applying condition $(2)$ of Theorem \ref{theorem A} with $F_n=F$ and  $\delta_n=\dfrac{\mu(F)}{4}$, we get  sequences $(s_n)_{n\ge1}\subset S$, $(\lambda_n)_{n\ge1}\subset\Gamma\setminus\{0\}$ and a sequence of compact sets $K_n\subset F$ such that $\mu(F\setminus K_n)<\delta_n$ and
\begin{equation}
\lim_{n\to+\infty} |\lambda_n|\|\omega\|_{p,s_ns_{1}^{-1}K_1}=0
\quad\text{ and }\quad
\lim_{n\to+\infty} \dfrac{1}{|\lambda_n|}\|\omega\|_{p,s_1s_{n}^{-1}K_n}=0.
\label{21}
\end{equation}
In particular, for every $n\ge1$ we have $\mu(K_n\cap K_1)>\dfrac{\mu(F)}{2}$. Let $C>0$ be such that $M(s_1^{-1})\le C$ and $R|_{K_{1}^{-1}\cup K_{1}^{-1}s_1}\le C$. 
Then
$$
\omega(s_n) = \inf_{t\in s_1^{-1}K_1} \omega(s_n t t^{-1})
 \le C \inf_{t\in s_1^{-1} K_1} \omega(s_n t) \le C\|\omega\|_{p,s_{n}s_{1}^{-1}K_1}\mu(K_1)^{-1/p}
$$
and
\begin{align*}
\omega(s_n^{-1}) & = \inf_{t\in K_1\cap K_n} \omega( s_1^{-1} s_1 s_n^{-1}t t^{-1} )
 \le C^2 \inf_{t\in K_1\cap K_n} \omega( s_1 s_n^{-1}t )
\\&\le C^{2}\|\omega\|_{p,s_1s_n^{-1}K_n}\mu(K_1\cap K_n)^{-1/p}<2C^2\mu(F)^{-1/p}\|\omega\|_{p,s_1s_n^{-1}K_n}.
\end{align*}
Combining the last inequalities with \eqref{21}, we obtain
$$
\max\{ |\lambda_n|\omega(s_n)\,;\,\frac1{|\lambda_n|}\omega(s_n^{-1})\}\underset{n\to+\infty}{\longrightarrow}0.
$$
\end{proof}

We can now easily obtain the following corollary, which provides for several types of $\Gamma$ a complete characterization of $(\Gamma,S)$-density only in terms of the weight.
\begin{cor}\label{Gamma-bd}
Suppose that $(G,S,p,\omega,\Gamma)$ is an admissible tuple. If $\omega$ is a continuous weight such that for every $s\in G$,
$$
M(s):=\underset{t\in G}{\sup}\,\dfrac{\omega(st)}{\omega(t)}<+\infty \quad\text{and}\quad 
R(s):=\underset{t\in G}{\sup}\,\dfrac{\omega(ts)}{\omega(t)}<+\infty,
$$
then the following conditions hold: 
\begin{enumerate}[label={$(\arabic*)$}]
\item  If $\Gamma\setminus\{0\}$ is bounded and bounded away from zero, there is a $(\Gamma,S)$-dense vector in  $L^p(G,\omega)$ if and only if $$\inf\limits_{s\in S}\max\{ \omega(s);\omega(s^{-1})\}=0.$$
\item There is a $([0,1],S)$-dense vector in  $L^p(G,\omega)$ if and only if \[\inf\limits_{s\in S} \max\{ \omega(s)\omega(s^{-1});\omega(s^{-1})\}=0.\]
\item There is a $([1,+\infty[,S)$-dense vector in  $L^p(G,\omega)$ if and only if \[\inf\limits_{s\in S} \max\{ \omega(s)\omega(s^{-1});\omega(s)\}=0.\]
\item There is a $(\C,S)$-dense vector in  $L^p(G,\omega)$ if and only if $\inf\limits_{s\in S} \omega(s)\,\omega(s^{-1})=0$.
\end{enumerate}
\label{29}
\end{cor}
\begin{proof}  In view of Proposition \ref{MR}, the case (1) is easy to check. Suppose that the condition on $\omega$ in (2) holds. For every $k\geqslant 1$ there exists $s_k\in S$ such that 
\[
\omega(s_k)\omega(s_k^{-1})\leqslant\frac1{k^2} \quad\text{ and }\quad \omega(s_k^{-1})\leqslant\frac1{k^2}.
\]
If we set $\lambda_k=k\,\omega(s_{k}^{-1})$ then $0\leqslant\lambda_k\leqslant\dfrac{1}{k}\leqslant1$ and
\[
\lambda_k\omega(s_k)=k\,\omega(s_k)\omega(s_{k}^{-1})\leqslant\dfrac{1}{k},
\qquad \frac1{\lambda_k}\omega(s_{k}^{-1})=\frac1k.
\]
Thanks to Proposition \ref{25}, $L^p(G,\omega)$ has a $([0,1],S)$-dense vector. The converse follows easily from Proposition \ref{25}. The proof of (3) is similar.
For (4), by Proposition \ref{MR}, the condition on $\omega$ is necessary. Assume now that $\underset{s\in S}{\inf} \omega(s)\,\omega(s^{-1})=0$. For every $k\geqslant 1$ there exists $s_k\in S$ such that 
\[\omega(s_k)\omega(s_{k}^{-1})\leqslant\dfrac{1}{k^2}.\]
Set $\lambda_k=\omega(s_{k}^{-1})^{1/2}\omega(s_k)^{-1/2}$, then
\[\lambda_k\omega(s_k)=\omega(s_{k}^{-1})^{1/2}\omega(s_k)^{1/2}\leqslant\dfrac{1}{k}\underset{k\to\infty}{\longrightarrow}0\]
and
\[\dfrac{1}{\lambda_k}\omega(s_{k}^{-1})=\omega(s_{k}^{-1})^{1/2}\omega(s_k)^{1/2}\leqslant\dfrac{1}{k}\underset{k\to\infty}{\longrightarrow}0,\]
so that by Proposition \ref{25}, $L^p(G,\omega)$ has a $(\mathbb{C},S)$-dense vector.
\end{proof}

The following example shows that we do need all translations bounded to have the equivalence (4) in Corollary \ref{29}. The following weight is $S$-admissible, but not $G$-admissible.
\begin{ex}
Let $G=\Z$, $S=\Z\setminus\N$ and let $\omega$ be the weight defined on $G$ by \[ \omega_{n}=\begin{cases}
 2^{2^n} &\text{ if } n\geqslant0\\
 2^{-2^{-n+1}}  &\text{ if } n<0
\end{cases}.
\]
Then, we have
\[\underset{n\in S}{\inf}\, \omega_{n}\omega_{-n}=0 \quad\text{and }\quad\underset{k\in \mathbb{Z}}{\sup}\,\dfrac{\omega_{n+k}}{\omega_k}<+\infty\quad\text{ for all }n\in S,\]
so that the condition (4) of Corollary \ref{29} holds. But $L^p(G,\omega)$ doesn't have a $(\C,S)$-dense vector,  since $\underset{n\to+\infty}{\liminf}\,\omega_{n+1}\omega_{-n+1}\neq0$ (see \cite[Theorem 1.38]{BM}). 
\end{ex}

\section{Commutative subgroups}\label{sec-abelian}

We keep the same assumptions on $G,S,p,\Gamma,\omega$ as before, see Definition \ref{assump}.
If the subgroup generated by $S$ is abelian, Theorem B below shows that the converse implication of Theorem \ref{18} is also true. For convenience, we will recall Theorem \ref{theorem B} with an additional equivalent condition.

\begin{Theorem B}
Suppose that $(G,S,p,\omega,\Gamma)$ is admissible. If the \textbf{subgroup generated by $S$ is abelian}, then the following conditions are equivalent:
\begin{enumerate}[label={$(\arabic*)$}]
\item There is a $(\Gamma,S)$-dense vector in  $L^p(G,\omega)$.
\item For any compact subset $F\subset G$ and $\e>0$, there are $s\in S$, $\lambda\in\Gamma\setminus\{0\}$ and a compact subset $E\subset F$ such that $\mu(F\setminus E)<\e$ and
\[
|\lambda|\, \underset{t\in E}{\mathrm{ess }\sup }\,\omega(st)<\e\quad\text{ and }\quad
 \frac1{|\lambda|}\, \underset{t\in E}{\mathrm{ess }\sup}\,\omega(s^{-1}t)<\e.
\]
\item 
For any compact subset $F\subset G$ and  $\e>0$, there are $s\in S$, $\lambda\in\Gamma\setminus\{0\}$ and a compact subset $E\subset F$ such that $\mu(F\setminus E)<\e$ and
\begin{equation}
|\lambda|\,\|\omega\|_{p,sE}<\e \quad\text{ and }
\quad \frac1{|\lambda|}\mkern1mu \|\omega\|_{p,s^{-1}E}<\e.
\end{equation}
\end{enumerate}
\label{2}
\end{Theorem B}
\begin{proof}[Proof of Theorem B]
According to Theorem \ref{18}, $(2)\Rightarrow(1)$ holds for any set $S$ and $S$-admissible weight. 

Let us show that $(1)\Rightarrow(3)$. Let $F$ be a compact subset of $G$ and $\e>0$. It is clear that we can assume that $F$ has a positive measure. For every $n\ge1$, set $F_n=F$ and $\delta_n=2^{-n}$. By Theorem \ref{theorem A}, there are $(s_n)_{n\ge1}\subset S$, $(\lambda_n)_{n\ge1}\subset\Gamma\setminus\{ 0\}$ and $E_n\subset F$ compact such that $\mu(F\setminus E_n)<\delta_n$ and
\[
\sum_{n,k\ge0;n\neq k}\dfrac{|\lambda_n|^p}{|\lambda_k|^p}\|\omega\|_{p,s_ns_{k}^{-1}E_{k}}^p<+\infty
\]
with $s_0=e$, $\lambda_0=1$ and $E_0=\emptyset$. In particular,
\begin{equation}
\underset{k\to+\infty}{\lim}\dfrac{1}{|\lambda_k|^p}\|\omega\|_{p,s_{k}^{-1}E_{k}}^p=0.
\label{19}
\end{equation}
Let us fix $k\ge1$ such that $\delta_k<\dfrac{\e}{2}$ and set $C:=\|T_{s_k}\|$. We have
\begin{align*}
\|\omega\|_{p,s_nE_k}^p=\int_{s_ks_ns_{k}^{-1}E_k}\omega(t)^p\mathrm{d}\mu(t)\le C^p\|\omega\|_{p,s_ns_{k}^{-1}E_k}^p
\end{align*}
Since $|\lambda_n|\|\omega\|_{p,s_ns_{k}^{-1}E_k}\underset{n\to+\infty}{\longrightarrow}0$, it follows
\begin{equation}
|\lambda_n|\|\omega\|_{p,s_nE_k}\underset{n\to+\infty}{\longrightarrow}0.
\label{20}
\end{equation}
By \eqref{19} and \eqref{20} there exists $n>k$ such that
\[
\frac1{|\lambda_n|}\|\omega\|_{p,s_n^{-1}E_n}<\e \quad\text{ and }\quad
|\lambda_n|\|\omega\|_{p,s_nE_k}<\e.
\]
If we set $E=E_k\cap E_n$, $s=s_n$ and $\lambda=\lambda_n$, we have $\mu(F\setminus E)\le \mu(F\setminus E_k) +\mu(F\setminus E_n)<\delta_k+\delta_n\le 2\delta_k<\e$, and
\[  
|\lambda|\,\|\omega\|_{p,sE}<\e \quad\text{ and }\quad
\frac1{|\lambda|}\|\omega\|_{p,s^{-1}E} <\e.
\]

Let us show now that $(3)\Rightarrow(2)$. Let $F\subset G$ be a compact subset of $G$ and $\e>0$. Let $\eta>0$ be such that $4\eta<\e$. There exist $s\in S$, $\lambda\in\Gamma\setminus\lbrace0\rbrace$ and a compact subset  $K\subset F$ such that $\mu(F\setminus K)<\eta$ and 
$$
|\lambda| \|\omega\|_{p,sK} <\eta^{1+\frac1p}, \qquad \frac1{|\lambda|}\|\omega\|_{p,s^{-1}K}<\eta^{1+\frac1p}.
$$
Set
$$
E_1=\{ t\in K:\, \omega(st)\le \frac\eta{|\lambda|} \}, \qquad
E_2=\{ t\in K:\, \omega(s^{-1}t)\le \eta|\lambda|\},
$$
then
$$
\|\omega\|_{p,sK}^p \ge \frac{\eta^p}{|\lambda|^p} \mu(K\setminus E_1) \quad \text{and }\quad
 \|\omega\|_{p,s^{-1}K}^p \ge \eta^p|\lambda|^p \mu(K\setminus E_2).
$$
Hence $\mu(K\setminus E_1)<\eta$ and $\mu(K\setminus E_2)<\eta$, so that
$$
\mu(K\setminus (E_1\cap E_2))<2\eta.
$$
Let now $E\subset E_1\cap E_2$ be a compact subset such that $\mu((E_1\cap E_2)\setminus E)<\eta$.
Then $\mu(F\setminus E)<4\eta<\e$,
and
$$
|\lambda|\, \underset{t\in E}{\mathrm{ess }\sup }\,\omega(st)\le \eta<\e,\qquad
 \frac1{|\lambda|}\, \underset{t\in E}{\mathrm{ess }\sup}\,\omega(s^{-1}t)\le \eta <\e.
$$
\end{proof}

\begin{rem}
The condition (2) is independent of $p$, which means that it applies to any $p$. We have however to make sure that the usual condition $\omega\in L^p_{loc}(G)$ is satisfied, and this one does depend on $p$.
\end{rem}

For special types of $\Gamma$, one can characterize $(\Gamma, S)$-density only in terms of the weight, similarly to Corollary \ref{Gamma-bd}.

\begin{cor}
Suppose that $(G,S,p,\omega,\Gamma)$ is an admissible tuple. If the \textbf{subgroup generated by $S$ is abelian}, then:
\begin{enumerate}[label={$(\arabic*)$}]
\item  There is an $S$-dense vector in  $L^p(G,\omega)$ if and only if for any compact subset $F\subset G$ and $\e>0$, there are $s\in S$ and a compact subset $E\subset F$ such that $\mu(F\setminus E)<\e$ and
\[\underset{t\in E}{\mathrm{ess }\sup }\,\omega(st)<\e,\qquad \underset{t\in E}{\mathrm{ess }\sup}\,\omega(s^{-1}t)<\e.\]
\item  There is a $([0,1],S)$-dense vector in  $L^p(G,\omega)$ if and only if for any compact subset $F\subset G$ and $\e>0$, there are $s\in S$ and a compact subset $E\subset F$ such that $\mu(F\setminus E)<\e$ and
\[
\underset{t\in E}{\mathrm{ess }\sup}\,\omega(s^{-1}t)<\e, \qquad
\underset{t\in E}{\mathrm{ess }\sup }\,\omega(st)\, \underset{t\in E}{\mathrm{ess }\sup}\,\omega(s^{-1}t)<\e.
\]
\item There is a $([1,+\infty[,S)$-dense vector in  $L^p(G,\omega)$ if and only if for any compact subset $F\subset G$ and $\e>0$, there are $s\in S$ and a compact subset $E\subset F$ such that $\mu(F\setminus E)<\e$ and
\[
\underset{t\in E}{\mathrm{ess }\sup }\,\omega(st)<\e, \qquad
\underset{t\in E}{\mathrm{ess }\sup }\,\omega(st)\,\underset{t\in E}{\mathrm{ess }\sup}\,\omega(s^{-1}t)<\e.
\]
\item  There is a $(\mathbb{C},S)$-dense vector in  $L^p(G,\omega)$ if and only if  for any compact subset $F\subset G$ and $\e>0$, there are $s\in S$ and a compact subset $E\subset F$ such that $\mu(F\setminus E)<\e$ and
\[
\underset{t\in E}{\mathrm{ess }\sup }\,\omega(st)\, \underset{t\in E}{\mathrm{ess }\sup}\,\omega(s^{-1}t)<\e.
\]
\end{enumerate}
\label{26}
\end{cor}
\begin{proof} The assertion $(1)$ is a part of \cite[Theorem 10]{AK}. In (2), if there is a $([0,1],S)$-dense vector, then, as $1/|\lambda|\ge1$, the condition on $\omega$ follows from (2) of Theorem \ref{theorem B}.
Let us prove the converse. Suppose that the condition on $\omega$ in (2) holds; 
we will check the condition (2) from Theorem \ref{theorem B}.
 Let $F\subset G$ be a compact set. If $\mu(F)=0$, the inequalities in (2) of Theorem \ref{theorem B} hold with $E=\emptyset$ and any $s$ and $\lambda$. We can assume therefore that $\mu(F)>0$. For every $k\geqslant 1$ there exist $s_k\in S$ and a compact subset $E_k\subset F$ such that $\mu(F\setminus E_k)\leqslant\frac1k$
\[
\underset{t\in E_k}{\mathrm{ess }\sup}\,\omega(s_k^{-1}t)<\frac1{k^2} \quad\text{ and }\quad
\underset{t\in E_k}{\mathrm{ess }\sup }\,\omega(s_kt)\, 
\underset{t\in E_k}{\mathrm{ess }\sup}\,\omega(s_k^{-1}t)<\frac1{k^2}.
\]
Set $c_{k}=\underset{t\in E_k}{\mathrm{ess }\sup }\,\omega(s_kt)$, $d_k=\underset{t\in E_k}{\mathrm{ess }\sup}\,\omega(s_{k}^{-1}t)$. For $k$ big enough, the measure of $E_k$ is positive so that $c_k>0$ and $d_k>0$. Set $\lambda_k=k d_{k}$ (If $d_k=0$, which may happen for a finite number of $k$, we set $\lambda_k=1$). Thus for each $k$, we have $\lambda_k\in(0,1]$,
\[
\lambda_k c_k\underset{k\to\infty}{\longrightarrow}0 \quad\text{and}\quad
\frac1{\lambda_k}d_k\underset{k\to\infty}{\longrightarrow}0.
\]
Thanks to the assertion (2) of Theorem \ref{theorem B}, $L^p(G,\omega)$ has a $([0,1],S)$-dense vector. The proof of (3) is similar. For (4), the direct implication follows immediately from (2) of Theorem \ref{theorem B}. Let us prove  that the condition on $\omega$ in (4) implies that $L^p(G,\omega)$ has a $(\mathbb{C},S)$-dense vector. We will show that the assertion (2) of Theorem \ref{theorem B} holds. Let $F\subset G$ be a compact set, and as in (2), we can assume that $\mu(F)>0$. For every $k\geqslant 1$ there exist $s_k\in S$ and a compact subset $E_k\subset F$ such that $\mu(F\setminus E_k)\leqslant\dfrac{1}{k}$ and 
\[\underset{t\in E_k}{\mathrm{ess }\sup }\,\omega(s_kt)\, \underset{t\in E_k}{\mathrm{ess }\sup}\,\omega(s_{k}^{-1}t)<\dfrac{1}{k}.\]
Set $c_{k}=\underset{t\in E_k}{\mathrm{ess }\sup }\,\omega(s_kt)$, $d_k=\underset{t\in E_k}{\mathrm{ess }\sup}\,\omega(s_{k}^{-1}t)$ and $\lambda_k=d_{k}^{1/2}c_{k}^{-1/2}$ (or $\lambda_k=1$ if $c_k=0$; as in (2), this may happen only for a finite number of $k$). Now
\[\lambda_k c_k\underset{k\to\infty}{\longrightarrow}0 \quad\text{ and } \quad \dfrac{1}{\lambda_k}d_k\underset{k\to\infty}{\longrightarrow}0.\]
Thanks to the assertion (2) of Theorem \ref{theorem B}, $L^p(G,\omega)$ has a $(\mathbb{C},S)$-dense vector.
\end{proof}

{\textbf{Acknowledgements.}}
 The first named author gratefully acknowledges many helpful suggestions, comments and remarks of his supervisors Evgeny Abakumov and St\'{e}phane Charpentier during the preparation of this paper.
 
This work was partially supported by the French ``Investissements d'Avenir'' program, project ISITE-BFC (contract ANR-15-IDEX-03).

\end{document}